\documentclass[11pt,a4paper]{amsart}
\pdfoutput=1
\usepackage[margin=2.5cm]{geometry} 
\usepackage{amssymb}
\usepackage{amsmath}
\usepackage{mathtools}
\usepackage{amsthm}
\usepackage{amsfonts}
\usepackage{kpfonts} 
\usepackage{bbm} 
\usepackage{dutchcal} 
\usepackage[mathscr]{euscript} 
\usepackage{tikz-cd} 
\usepackage{todonotes} 
\usepackage{graphicx} 
\usepackage[toc,page]{appendix} 
\usepackage{xcolor} 
\usepackage{float} 
\usepackage{enumitem} 
\usepackage{comment} 
\usepackage[normalem]{ulem} 

\usepackage{url}
\usepackage[style = alphabetic,
            backref = true,
            backrefstyle = two]{biblatex}
\renewcommand\mathfrak[1]{\mbox{\usefont{U}{euf}{m}{n}#1}}

\definecolor{winered}{rgb}{0.5,0,0}
\usepackage[citecolor = winered, 
            urlcolor = winered, 
            linkcolor = winered,
            menucolor = winered, 
            colorlinks = true]{hyperref}

\DefineBibliographyStrings{english}{%
  backrefpage = {p.}, 
  backrefpages = {pp.}, 
}

\newtheoremstyle{theoremdd}
{\topsep}{\topsep}{\upshape}{0pt}{\bfseries}{.}{ }{\thmname{#1}\thmnumber{ #2}\thmnote{ (#3)}}

\theoremstyle{definition}
\newtheorem{Th}{Theorem}[section]
\newtheorem{Lemma}[Th]{Lemma}
\newtheorem{Cor}[Th]{Corollary}
\newtheorem{Prop}[Th]{Proposition}
\newtheorem{Def}[Th]{Definition}

\newtheorem{Rem}[Th]{Remark}
\newtheorem{Ex}[Th]{Example}

\newcommand{\Hom}{\text{Hom}}

\newcommand{\R}{\mathbb{R}}
\newcommand{\Z}{\mathbb{Z}}

\newcommand{\ncat}{\mathbf} 
\newcommand{\cat}{\mathcal} 
\renewcommand{\u}{\underline}
\newcommand{\colim}{\text{colim}}
\newcommand{\ncolim}[1]{\underset{#1}{\colim} \,}

\newcommand{\op}{\text{op}}
\newcommand{\Gr}{\ncat{Gr}}

\newcommand{\Sing}{\text{Sing}}
\renewcommand{\Re}{\text{Re}}
\newcommand{\oSing}{\text{oSing}}

\newcommand{\tSing}{\widetilde{\text{S}}\text{ing}}
\newcommand{\tRe}{\widetilde{\text{R}}\text{e}}
\newcommand{\Sd}{\u{\text{Sd}}}
\newcommand{\sd}{\text{Sd}}
\newcommand{\Cl}{\text{C}\ell}
\newcommand{\st}{\text{st}}

\newcommand{\Face}{\text{Face}}
\newcommand{\Ord}{\text{Ord}}
\renewcommand{\H}{\mathbb{H}}
\newcommand{\bDelta}{\u{\Delta}}

\makeatletter
\newtheorem*{rep@theorem}{\rep@title}
\newcommand{\newreptheorem}[2]{%
\newenvironment{rep#1}[1]{%
 \def\rep@title{#2 \ref{##1}}%
 \begin{rep@theorem}}%
 {\end{rep@theorem}}}
\makeatother

\newreptheorem{theorem}{Theorem}

\newreptheorem{lemma}{Lemma}

\newreptheorem{cor}{Corollary}
\newreptheorem{prop}{Proposition}

\usetikzlibrary{calc}
\usetikzlibrary{decorations.pathmorphing}

\tikzset{curve/.style={settings={#1},to path={(\tikztostart)
    .. controls ($(\tikztostart)!\pv{pos}!(\tikztotarget)!\pv{height}!270:(\tikztotarget)$)
    and ($(\tikztostart)!1-\pv{pos}!(\tikztotarget)!\pv{height}!270:(\tikztotarget)$)
    .. (\tikztotarget)\tikztonodes}},
    settings/.code={\tikzset{quiver/.cd,#1}
        \def\pv##1{\pgfkeysvalueof{/tikz/quiver/##1}}},
    quiver/.cd,pos/.initial=0.35,height/.initial=0}

\tikzset{tail reversed/.code={\pgfsetarrowsstart{tikzcd to}}}
\tikzset{2tail/.code={\pgfsetarrowsstart{Implies[reversed]}}}
\tikzset{2tail reversed/.code={\pgfsetarrowsstart{Implies}}}

\definecolor{emilioeditcolor}{rgb}{0.94, 0.97, 1.0}

\title{Thomason-Type Model Structures on Simplicial Complexes and Graphs}
\author{Emilio Minichiello}

\usepackage[foot]{amsaddr}
\address{CUNY CityTech}
\email{emilio.minichielloepstein04@citytech.cuny.edu, \href{www.emiliominichiello.com}{emiliominichiello.com}.}

\begin{document}

\begin{abstract}
In this paper we show that the Matsushita model structure on loop graphs, which is right-transferred from the Kan-Quillen model structure on simplicial sets, factors through two other right-transferred model structures on simplicial complexes and reflexive graphs. We show that each Quillen adjunction between these right-transferred model categories is a Quillen equivalence. These model structures are analogous to the Thomason model structure on small categories, and we prove that they are all cofibrantly generated and proper. Furthermore we show that all cofibrant simplicial complexes are flag complexes, and all forests are cofibrant.
\end{abstract}

\maketitle

\setcounter{tocdepth}{1}
\tableofcontents

\section{Introduction}

In recent years, the notion of graph homotopy has gained popularity. In fact, two distinct graph homotopy theories have emerged: the $A$-homotopy theory and the $\times$-homotopy theory. These homotopy theories are most conveniently understood on the category $\ncat{Gr}$ of reflexive, undirected, simple graphs (Definition \ref{def reflexive graph}). Kapulkin and Kershaw have recently shown \cite{kapulkin2024closed} that on this category, there exist precisely two symmetric monoidal closed structures. They are given by the monoidal product $\square$, called the \textbf{box product} or \textbf{cartesian product}, and $\times$, called the \textbf{categorical product}. To each of these symmetric monoidal closed products $\odot \in \{\square, \times \}$ we can associate a notion of homotopy by studying maps $I_n \to \Hom^\odot(G,H)$, where $I_n$ denotes a path of length $n$ and $\Hom^\odot(G,-)$ is the right adjoint to the functor $G \odot -$. The box product and categorical product correspond to the $A$-theory and the $\times$-theory respectively. Rather than study a graph $G$ as a topological space, in which case its homotopy type is rather trivial, these notions of homotopy capture deeper combinatorial information about graphs, see \cite{babson2006homotopy, barcelo2005perspectives, dochtermann2009hom, dochtermann2009homotopygroup}.

The $A$-homotopy theory of reflexive graphs was recently shown to be inherited from a cubical nerve functor \cite{kapulkin2022cubical}. However, it is currently unknown if the $A$-homotopy theory of reflexive graphs is equivalent to the homotopy theory of spaces, though there has been some recent progress \cite{kapulkin2024homotopyntypescubicalsets}. 

Meanwhile, the $\times$-homotopy theory of loop graphs (Definition \ref{def loop graph}) is inherited from the homotopy theory of clique complexes. The origins of this theory lie in the Hom-complex construction of Lov\'{a}sz, which he used to prove the Kneser conjecture in \cite{lovasz1978kneser}. Given two graphs $G$ and $H$, this is a space\footnote{There are actually many different models for the homotopy type of $\Hom(G,H)$. Lov\'{a}sz used the neighborhood complex, see \cite[Section 17]{kozlov2008combinatorial}.} $\Hom(G,H)$ whose topological invariants are tightly related to various combinatorial invariants of $G$ and $H$. In particular, in the case $G = K^2$, topological invariants of the $\Z_2$-equivariant homotopy type of $\Hom(G,H)$ give lower bounds on the chromatic number of $H$, see \cite{daneshpajouh2025box} for a nice overview of this connection. Dochtermann showed in \cite[Remark 3.6]{dochtermann2009hom} that there is a homotopy equivalence
\begin{equation*}
    \Hom(G,H) \simeq \Cl((H^G)^\circ)
\end{equation*}
where the right-hand side is the clique complex of the maximal reflexive subgraph of the exponential graph $H^G$. 

Motivated by the desire to better understand these Hom-complexes of graphs, the author searched the literature for abstract homotopical structures relevant to $\times$-homotopy theory. Two such structures are given by Matsushita in \cite{matsushita2017box}. In more detail, Matsushita constructs two model structures on the category $\ncat{Gr}_\ell$ of loop graphs, one of which is Quillen equivalent to the Kan-Quillen model structure on $\ncat{sSet}$, and one of which is Quillen equivalent to a model structure on $\Z_2\ncat{sSet}$, simplicial sets with a $\Z_2$-action. For this paper, we will focus on the former, which we will call the Matsushita model structure on $\ncat{Gr}_\ell$. Matsushita constructs this model structure by right-transferring the Kan-Quillen model structure to $\ncat{Gr}_\ell$ along an adjunction. This results in a model structure whose weak equivalences (fibrations) are precisely those graph homomorphisms $ f: G \to H$ which induce a weak equivalence (fibration) on the clique complexes of their maximal reflexive subgraphs.

There is something curious about this model structure. There exists an adjunction 
\begin{equation*}
    \begin{tikzcd}
	{\ncat{Gr}_\ell} && {\ncat{sSet}}
	\arrow[""{name=0, anchor=center, inner sep=0}, "{\Sing_{\Gr_\ell}}"', shift right=3, from=1-1, to=1-3]
	\arrow[""{name=1, anchor=center, inner sep=0}, "{\Re_{\Gr_{\ell}}}"', shift right=3, from=1-3, to=1-1]
	\arrow["\dashv"{anchor=center, rotate=-90}, draw=none, from=1, to=0]
\end{tikzcd}
\end{equation*}
between graphs and simplicial sets, where $\Sing_{\ncat{Gr}_\ell}$ is a simplicial set version of the clique complex of the maximal reflexive subgraph. However, this is not the adjunction that Matsushita transfers the model structure along. Rather, he uses the adjunction 
\begin{equation*}
\begin{tikzcd}
	{\ncat{Gr}_\ell} && {\ncat{sSet}}
	\arrow[""{name=0, anchor=center, inner sep=0}, "{\text{Ex}^2 \, \Sing_{\Gr_\ell}}"', shift right=3, from=1-1, to=1-3]
	\arrow[""{name=1, anchor=center, inner sep=0}, "{\Re_{\Gr_{\ell}} \, \sd^2}"', shift right=3, from=1-3, to=1-1]
	\arrow["\dashv"{anchor=center, rotate=-90}, draw=none, from=1, to=0]
\end{tikzcd}
\end{equation*}
where $\sd$ is the barycentric subdivision functor on simplicial sets and $\text{Ex}$ is its right adjoint.

Now this somewhat mysterious second subdivision functor is well known to homotopy theorists. Indeed, in 1980, Thomason published the paper ``Cat as a closed model category'' \cite{thomason1980cat} in which he proved that the right-transferred model structure of the Kan-Quillen model structure along the adjunction
\begin{equation*}
\begin{tikzcd}
	{\ncat{Cat}} && {\ncat{sSet}}
	\arrow[""{name=0, anchor=center, inner sep=0}, "{{\text{Ex}^2 \, N}}"', shift right=3, from=1-1, to=1-3]
	\arrow[""{name=1, anchor=center, inner sep=0}, "{{h \, \sd^2}}"', shift right=3, from=1-3, to=1-1]
	\arrow["\dashv"{anchor=center, rotate=-90}, draw=none, from=1, to=0]
\end{tikzcd}  
\end{equation*}
exists, where $N$ denotes the nerve functor and $h$ its left adjoint. This model structure on the category $\ncat{Cat}$ of small categories is called the \textbf{Thomason model structure}. The main geometric idea behind this construction is the fact that for any simplicial set $X$, the counit $N h X \to X$ is not necessarily a weak equivalence, but for any simplicial set $X$, the canonical map $N h \Sd^2 X \to X$ is in fact a weak equivalence of simplicial sets. This amazing fact was proven independently by Thomason and Fritsch-Latch in \cite{fritsch1981homotopy}. Using this, it can be shown that the adjunction above is a Quillen equivalence.

A weak equivalence in the Thomason model structure is a functor $F : \cat{C} \to \cat{D}$ such that $NF : N \cat{C} \to N \cat{D}$ is a weak equivalence of simplicial sets. There are some interesting trade-offs regarding the Thomason model structure. Its weak equivalences are simple to understand, it is cofibrantly generated and proper. Furthermore, we have a very good understanding of the homotopy colimits in the Thomason model structure via the Grothendieck construction \cite{thomason1979homotopy}, and the construction of homotopy fibers via Quillen's Theorems A and B \cite{quillen1972higher}. However, a downside of this model structure is that we do not have a complete understanding of its fibrant and cofibrant objects. We know that all cofibrant categories are posets by \cite[Proposition 5.7]{thomason1980cat}. Many partial results have been proved for cofibrant objects in \cite{bruckner2016cofibrant} and for fibrant objects in \cite{meier2015fibrancy}.

The author was surprised to learn of the variety of Thomason-type model structures in the literature that have cropped up in recent years: on posets \cite{raptis2010homotopy}, on acyclic categories \cite{bruckner2015acyclic}, on relative categories \cite{barwick2012relative}, on nonsingular simplicial sets \cite{fjellbo2020nonsingular}, on $G$-categories \cite{Bohmann_2015} where $G$ is a finite group, on $n$-fold categories \cite{fiore2010thomason}, on $2$-categories modelling the homotopy theory of spaces \cite{ara2015structures}, on $2$-categories \cite{pavlov2024enriched} modelling the homotopy theory of $\infty$-categories, and on $\infty$-categories \cite{mazel2019grothendieck} modelling the homotopy theory of spaces.

Shortly after the preprint version of this paper was posted to the arXiv, the preprint \cite{wei2025ordered} was posted, which produces a Thomason-type model structure on the category of ordered simplicial complexes.

We could not find analogous results for the category of simplicial complexes $\ncat{Cpx}$ or the category of reflexive graphs $\ncat{Gr}$ in the literature. In this paper we obtain such results, showing that Matsushita's construction actually factors through two intermediate model categories

\begin{equation} \label{eq total composite adjunction}
\begin{tikzcd}
	{\ncat{Gr}_\ell} && {\ncat{Gr}} && {\ncat{Cpx}} && {\ncat{sSet}}
	\arrow[""{name=0, anchor=center, inner sep=0}, "{(-)^\circ}"', shift right=3, from=1-1, to=1-3]
	\arrow[""{name=1, anchor=center, inner sep=0}, "{i_\ell}"', shift right=3, from=1-3, to=1-1]
	\arrow[""{name=2, anchor=center, inner sep=0}, "\Cl"', shift right=3, from=1-3, to=1-5]
	\arrow[""{name=3, anchor=center, inner sep=0}, "{(-)_{\leq 1}}"', shift right=3, from=1-5, to=1-3]
	\arrow[""{name=4, anchor=center, inner sep=0}, "{{{\text{Ex}^2 \, \text{Sing}}}}"', shift right=3, from=1-5, to=1-7]
	\arrow[""{name=5, anchor=center, inner sep=0}, "{{{\text{Re} \, \sd^2}}}"', shift right=3, from=1-7, to=1-5]
	\arrow["\dashv"{anchor=center, rotate=-90}, draw=none, from=1, to=0]
	\arrow["\dashv"{anchor=center, rotate=-90}, draw=none, from=3, to=2]
	\arrow["\dashv"{anchor=center, rotate=-90}, draw=none, from=5, to=4]
\end{tikzcd}    
\end{equation}
each of which are right-transferred along the intermediate adjunctions. We prove a wealth of properties about these model structures:
\begin{itemize}
    \item the right-transferred model structure on $\ncat{Cpx}$, which we call the \textbf{Thomason model structure} on $\ncat{Cpx}$ exists, is cofibrantly generated (Theorem \ref{th transferred model structure from sset}) and proper (Proposition \ref{prop thomason model structure is proper}). Furthermore, all cofibrant simplicial complexes are flag complexes (Proposition \ref{prop cofibrant complexes are flag}), the rightmost adjunction above is a Quillen equivalence and the functor $\text{Ex}^2 \, \Sing$ is also left Quillen (Lemma \ref{lem tsing is left quillen}).
    \item the right-transferred model structure on $\ncat{Gr}$, which we call the \textbf{Matsushita model structure} on $\ncat{Gr}$ exists (Theorem \ref{th transfer from cpx to gr}), is cofibrantly generated and proper (Proposition \ref{prop model structure on gr is cofib gen and proper}). Furthermore, the middle adjunction above is a Quillen equivalence (Theorem \ref{th quillen equiv between cpx and gr}).
    \item the right-transferred model structure on $\ncat{Gr}_\ell$, which we call the \textbf{Matsushita model structure} on $\ncat{Gr}_\ell$ exists (Theorem \ref{th transfer from gr to loop gr}), is cofibrantly generated and proper (Proposition \ref{prop model structure loop graphs is proper}). Furthermore, the leftmost adjunction above is a Quillen equivalence (Theorem \ref{th quillen equiv between gr and loop gr}) and the functor $i_\ell$ is also right Quillen (Lemma \ref{lem adjunction between loop graphs is biQuillen}).
    \item we are able to characterize several classes of cofibrant and fibrant objects in these model structures. All double subdivided simplicial complexes are Thomason cofibrant (Corollary \ref{cor cofibrant complexes}), all $n$-simplices are both Thomason fibrant and cofibrant (Lemma \ref{lem simplices are cofibrant} and \ref{lem simplices are thomason fibrant}), all reflexive complete graphs are Matsushita fibrant (Lemma \ref{lem complete graphs are fibrant}), all reflexive $n$-path graphs are Matsushita cofibrant and the inclusion of an endpoint is a cofibration (Lemma \ref{lem paths are cofibrant}), all reflexive graphs which are forests are Matsushita cofibrant (Corollary \ref{cor trees are cofibrant}) and all $n$-cycles $C_{n}$ for $n \geq 3$ are Matsushita cofibrant (Lemma \ref{lem 4n-cycles are cofibrant}).
\end{itemize}
This provides a useful framework for analyzing the $\times$-homotopy theory of graphs, and extends it to simplicial complexes.

In Section \ref{section simplicial complexes} we review the theory of simplicial complexes and its relation with ordered simplicial complexes and simplicial sets. We extend the notions of $\times$-homotopy theory of graphs from \cite{dochtermann2009hom} to simplicial complexes and prove various results about $\times$-homotopy of simplicial complexes that may be of independent interest. In Section \ref{section transfer simplicial complexes} we prove that the Thomason model structure on $\ncat{Cpx}$ exists, and in Section \ref{section properties of thomason} we explore some of its properties, eventually proving it is proper. In Section \ref{section graphs} we prove that the Matsushita model structures on $\ncat{Gr}$ and $\ncat{Gr}_\ell$ exist and investigate their properties. In Section \ref{section derived hom} we study the derived Hom in $\ncat{Gr}_\ell$ and determine that it does not respect the underlying cartesian-closed structure. We provide Appendix \ref{section categories} for some background on the categories of simplicial complexes and ordered simplicial complexes.

\subsection*{Acknowledgement}
The author was supported by PSC-CUNY Grant TRADB 56-13. He would also like to thank the participants at the NYC category theory seminar for their comments on a talk given about this work.

\subsection{Notation and Assumptions}

\begin{itemize}
    \item We use write unnamed categories as $\cat{C}, \cat{D}, \dots$ and named categories as $\ncat{Set}, \ncat{Cat}, \dots$
    \item We will assume the existence of a Grothendieck universe. Sets belonging to this universe are called small, subsets of the universe are called large. We will assume all of our categories are locally small, i.e. they have small hom-sets (See \cite[Appendix A]{minichiello2025coverages} for an introduction to set-theoretic assumptions useful in category theory.).
    \item We will mainly be concerned with the following categories for this paper: 
    \begin{itemize}
        \item $\ncat{Gr}$, the category of reflexive graphs,
        \item $\ncat{Gr}_\ell$, the category of loop graphs,
        \item $\ncat{Cpx}$, the category of simplicial complexes,
        \item $\ncat{oCpx}$, the category of ordered simplicial complexes, and
        \item $\ncat{sSet}$, the category of simplicial sets.
    \end{itemize}
    \item In order to keep this paper to a reasonable length, we must assume the reader is comfortable with simplicial homotopy theory and model categories. For introductory references on the subject we recommend \cite{riehl2014categorical, riehl2020homotopical, hirschhorn2003model, hovey2007model, balchin2021handbook}.
    \item We will often cite modern, readable references for definitions and folklore results rather than attempt to find the oldest such source.
    \item Since we will deal with both simplicial sets and simplicial complexes in this paper, when we are denoting an object that could be either a simplicial set or simplicial complex, like $\Delta^k$, we will use an underline to denote the simplicial complex version. So the $k$-simplex as a simplicial complex will be denoted $\u{\Delta}^k$.
\end{itemize}

\section{Simplicial Complexes} \label{section simplicial complexes}
In this section we review the information about simplicial complexes we will need for the construction of the Thomason model structure, and extend the basic notions of $\times$-homotopy theory from graphs to simplicial complexes. In Appendix \ref{section categories}, we look at the category of (ordered) simplicial complexes in more detail.

\subsection{Geometry of Simplicial Complexes}

\begin{Def} \label{def cpx}
A simplicial complex $K$ consists of a set $V(K)$ and a collection of finite subsets (which we also denote by $K \subseteq 2^{V(K)}$), which is closed under inclusion. In other words, if $\sigma \in K$ and $\tau \subseteq \sigma$, then $\tau \in K$. We call $\sigma \in K$ a \textbf{simplex} or \textbf{face} of $K$. The \textbf{dimension} of a face $\sigma$ is $|\sigma| - 1$. The dimension of a simplicial complex $K$ is the highest dimension of its faces. A morphism $f : K \to L$ of simplicial complexes consists of a function $V(f): V(K) \to V(L)$ such that if $\sigma \in K$, then $V(f)(\sigma) \in V(L)$. We let $\ncat{Cpx}$ denote the category of simplicial complexes.
\end{Def}

We study the category $\ncat{Cpx}$ more deeply in Appendix \ref{section categories}.

\begin{Def}
Given a simplicial complex $K$, we say that a face $\tau \in K$ is a \textbf{facet} if there exists no $\sigma \in K$ with $\tau \subset \sigma$. We say that a simplicial complex $K$ is a \textbf{subcomplex} of a simplicial complex $L$ on the same vertex set if $K \subseteq L$. A \textbf{pair of simplicial complexes} $(L, K)$ consists of a subcomplex $K \subseteq L$.
\end{Def}

Let us now connect this abstract notion with topology.

\begin{Def}[{\cite[Section 1.1]{ruschoff2017lecture}}] \label{def realization of simplicial complexes}
For a simplicial complex $K$, let $\oplus_{x \in V(K)} \R \cong \R^{V(K)}$ denote the real vector space with basis $ \{e_x \}_{x \in V(K)}$. If $a \in \R^{V(K)}$, let $a_x$ denote the $e_x$ component of $a$, and
\begin{equation} \label{eq realization of complexes}
    |K|_{\ncat{Cpx}} = \left \{ a \in \R^{V(K)} \; | \; a_x \geq 0, \, \sum_{x \in V(K)} a_x = 1, \, \{x \in V(K) \, | a_x \neq 0 \} \in K \right \},
\end{equation}
where we equip this subset of $\R^{V(K)}$ with the final topology. See \cite[Definition 1.6]{ruschoff2017lecture} for more details. This construction defines a functor $|-|_{\ncat{Cpx}} : \ncat{Cpx} \to \ncat{Top}$, which we call the \textbf{geometric realization} functor on simplicial complexes.
\end{Def}

Famously, the functor $|-|_{\ncat{Cpx}}$ does not preserve finite products (See \cite[Section 3]{bergner2022simplicial} for more details). One way to fix this problem is to consider ordered simplicial complexes.

\begin{Def} \label{def ordered complexes}
An \textbf{ordered simplicial complex} $K_{\leq}$ consists of a simplicial complex $K$ along with a partial order $\leq$ on $V(K)$ such that for each simplex $\sigma \in K$, the restricted partial order $\leq|_{\sigma}$ is a total order. A map $f : K_{\leq} \to L_{\leq}$ of ordered simplicial complexes is a map of simplicial complexes that also preserves the partial order on vertices. Let $\ncat{oCpx}$ denote the category of ordered simplicial complexes, and let $U : \ncat{oCpx} \to \ncat{Cpx}$ denote the forgetful functor.
\end{Def}

\begin{Rem}
While this paper is mainly focused on unordered simplicial complexes, there a few key results on ordered simplicial complexes we will need. In particular, we need Corollary \ref{cor realization of subdivision of horns, simplices, boundaries} for Theorem \ref{th transferred model structure from sset} and Proposition \ref{prop hoequiv of realizations} for Proposition \ref{prop Sing preserves x-homotopy equivalence}, which itself is used throughout much of the results on simplicial complexes.
\end{Rem}

Note that products in $\ncat{oCpx}$ are different than those in $\ncat{Cpx}$, see Example \ref{ex ordered simplicial complex products}. If we define $|-|_{\ncat{oCpx}} : \ncat{oCpx} \to \ncat{Top}$ using the same formula (\ref{eq realization of complexes}), then we have the following result.

\begin{Lemma}[{\cite[Proposition 1.25]{ruschoff2017lecture}}] \label{lem geo real of ord complexes preserves finite products}
The functor $|-|_{\ncat{oCpx}}$ preserves finite products between finite ordered simplicial complexes.
\end{Lemma}

Historically, Lemma \ref{lem geo real of ord complexes preserves finite products} still did not make ordered simplicial complexes convenient enough to serve as the basis for combinatorial models of homotopy types. Simplicial sets became the preferred such notion. The reference \cite{bergner2022simplicial} provides good context for this idea.

Let $\ncat{sSet} = \ncat{Set}^{\mathsf{\Delta}^\op}$ denote the category of simplicial sets, and let $|-| : \ncat{sSet} \to \ncat{Top}$ denote the corresponding geometric realization functor, which preserves finite products between finite simplicial sets \cite[Proposition 2.60]{ruschoff2017lecture} \footnote{In fact, by restricting the codomain to compactly generated Hausdorff spaces, \cite[Proposition 2.4]{goerss2009simplicial} shows that $|-|$ preserves all finite limits}. We equip $\ncat{sSet}$ with the \textbf{Kan-Quillen model structure} \cite{quillen2006homotopical} where a map $f : X \to Y$ of simplicial sets is a
\begin{itemize}
    \item weak equivalence if $|f| : |X| \to |Y|$ induces an isomorphism on all homotopy groups,
    \item fibration if $f$ is a Kan fibration, i.e. right lifts against all maps $\Lambda^n_k \to \Delta^n$ for $0 \leq k \leq n$ and $n \geq 1$,
    \item cofibration if it is a monomorphism.
\end{itemize}

We let $\sd : \ncat{sSet} \to \ncat{sSet}$ denote the barycentric subdivision functor for simplicial sets \cite[Definition 3.50]{ruschoff2017lecture} and $\text{Ex} : \ncat{sSet} \to \ncat{sSet}$ its right adjoint.

\begin{Lemma}
The following adjunction
\begin{equation*}
\begin{tikzcd}
	{\ncat{sSet}} && {\ncat{sSet}}
	\arrow[""{name=0, anchor=center, inner sep=0}, "{\text{Ex}}"', shift right=3, from=1-1, to=1-3]
	\arrow[""{name=1, anchor=center, inner sep=0}, "{\text{Sd}}"', shift right=3, from=1-3, to=1-1]
	\arrow["\dashv"{anchor=center, rotate=-90}, draw=none, from=1, to=0]
\end{tikzcd}
\end{equation*}
is a Quillen adjunction.
\end{Lemma}

\begin{proof}
The functor $\text{Ex}$ preserves fibrations by \cite[Page 6, (3)]{guillou2006KanEx}. There is a natural transformation $j : 1_{\ncat{sSet}} \to \text{Ex}$ that is componentwise a trivial cofibration by \cite[Theorem 3.3]{moss2020another}. Hence by $2$-of-$3$, the $\text{Ex}$ functor preserves weak equivalences. Thus $\text{Ex}$ preserves trivial fibrations.
\end{proof}

\begin{Def} \label{def osing}
Let $\oSing : \ncat{oCpx} \to \ncat{sSet}$ denote the functor that assigns to an ordered simplicial complex $K_{\leq}$ the simplicial set $\oSing(K_{\leq})$ defined degreewise by
\begin{equation*}
    \oSing(K_{\leq})_n = \ncat{oCpx}(\Delta^n, K_{\leq})
\end{equation*}
where here $\Delta^n$ is the ordered simplicial complex with totally ordered vertex set $\{0 \leq 1 \leq \dots \leq n \}$ and where every subset is a simplex.
\end{Def}

We discuss $\oSing$ in more detail in Appendix \ref{section categories}. This functor respects geometric realization in the following sense.

\begin{Lemma}[{\cite[Proposition 2.59]{ruschoff2017lecture}}] 
Given an ordered simplicial complex $K_{\leq}$ there is a natural homeomorphism
\begin{equation*}
    |K_{\leq}|_{\ncat{oCpx}} \xrightarrow{\cong} |\oSing(K_{\leq})|.
\end{equation*}
\end{Lemma}

If $K$ is a simplicial complex and $\leq$ is some partial ordering on $V(K)$ making $K_{\leq}$ into an ordered simplicial complex, then by definition we have
\begin{equation*}
    |K|_{\ncat{Cpx}} = |K_{\leq}|_{\ncat{oCpx}}.
\end{equation*}
Hence by assigning an ordering to a simplicial complex we obtain a homeomorphism between its geometric realization as a simplicial complex and as a simplicial set. Let us now turn to subdivision.

\begin{Def} \label{def subdivision}
Given a simplicial complex $K$, let $\Face(K)$ denote the partially ordered set whose underlying set is $K$, the simplices of $K$, and where $ \sigma \leq \tau$ if and only if $\sigma \subseteq \tau$. We call this the \textbf{face poset} of $K$. This defines a functor $\Face : \ncat{Pos} \to \ncat{Cpx}$ from the category of posets to simplicial complexes. 

Conversely, given a poset $P$, let $\Ord(P)$ denote the simplicial complex with $V(\Ord(P)) = P$, and $\{x_0, \dots, x_n \} \in \Ord(P)$ if and only if $x_0 \leq x_1 \leq \dots \leq x_n$. We call this the \textbf{order complex} of $P$. This defines a functor $\Ord : \ncat{Pos} \to \ncat{Cpx}$.

We let $\Sd : \ncat{Cpx} \to \ncat{Cpx}$ denote the composite functor $\Sd = \Ord \circ \Face$. We call this the \textbf{(barycentric) subdivision functor}. Note that for any simplicial complex $K$, $\Sd \, K$ is the simplicial complex with $V(\Sd \, K) = K$ and $\{\sigma_0, \dots, \sigma_n \} \in \Sd \, K$ if and only if $\sigma_0 \subset \sigma_1 \subset \dots \subset \sigma_n$. There is a canonical partial ordering on the vertices $\Sd \, K$ induced by inclusion of faces making $\Sd \, K$ into an ordered simplicial complex. Hence subdivision extends canonically to a functor $\Sd : \ncat{Cpx} \to \ncat{oCpx}$.
\end{Def}

\begin{Rem}
One might be tempted to think that $\Ord$ and $\Face$ are adjoint, but they are not. In fact, $\Face$ does not have either a left or right adjoint, because it does not preserve finite limits or colimits.
\end{Rem}

If $K_\leq$ is an ordered simplicial complex, then there is a map
\begin{equation} \label{eq last vertex map}
    \lambda : \Sd K_{\leq} \to K_{\leq}
\end{equation}
called the \textbf{last vertex map}. A vertex in $\Sd K_\leq$ is a simplex in $K_\leq$. Since $K_\leq$ is ordered, the set of vertices of each simplex is totally ordered. The map $\lambda$ is defined on vertices as follows, if $i_0 < i_1 < \dots < i_n$ are the vertices of a simplex $\sigma \in K_\leq$ then $\lambda(\sigma) =i_n$. Now given a $k$-simplex $\sigma_0 \subset \sigma_1 \subset \dots \subset \sigma_k$ in $\Sd K_\leq$, its image under $\lambda$ lands in $\sigma_k$. Hence $\lambda$ is a map of ordered simplicial complexes.

There is an adjunction
\begin{equation} \label{eq re sing adjunction}
\begin{tikzcd}
	{\ncat{Cpx}} && {\ncat{sSet}}
	\arrow[""{name=0, anchor=center, inner sep=0}, "{{\text{Sing}}}"', shift right=3, from=1-1, to=1-3]
	\arrow[""{name=1, anchor=center, inner sep=0}, "{{\text{Re}}}"', shift right=3, from=1-3, to=1-1]
	\arrow["\dashv"{anchor=center, rotate=-90}, draw=none, from=1, to=0]
\end{tikzcd}
\end{equation}
where if $K$ is a simplicial complex, $\Sing(K)$ is the simplicial set defined degreewise by
\begin{equation} \label{eq sing def}
    \Sing(K)_n = \ncat{Cpx}(\bDelta^n, K)
\end{equation}
where $\bDelta^n$ is the simplicial complex with $V(\bDelta^n) = \{0, 1, \dots, n \}$ and $S(\bDelta^n) = 2^{V(\bDelta^n)}$. We call $\Sing(K)$ the \textbf{singular simplicial set} associated to a simplicial complex $K$. Let $E[n]$ denote the indiscrete category with $n+1$ objects. In other words, this is the category with precisely one morphism between every pair of objects. It is not hard to see that $\Sing(\bDelta^n) = N(E[n])$, where $N$ is the nerve functor $N : \ncat{Cat} \to \ncat{sSet}$.

\begin{Lemma} \label{lem sing is fully faithful}
The functor $\Sing : \ncat{Cpx} \to \ncat{sSet}$ is fully faithful.
\end{Lemma}

\begin{proof}
Given simplicial complexes $K$ and $L$, we have
\begin{equation*}
    \begin{aligned}
        \ncat{sSet}(\Sing(K), \Sing(L)) & \cong \ncat{sSet}\left(\ncolim{\Delta^n \to \Sing(K)} \, \Delta^n, \Sing(L) \right) \\
        & \cong \lim_{\Delta^n \to \Sing(K)} \ncat{sSet}(\Delta^n, \Sing(L)) \\
        & \cong \lim_{\bDelta^n \to K} \Sing(L)_n \\
        & \cong \lim_{\bDelta^n \to K} \ncat{Cpx}(\bDelta^n, L) \\
        & \cong \ncat{Cpx}\left(\ncolim{\bDelta^n \to K} \, \bDelta^n, L\right) \\
        & \cong \ncat{Cpx}(K, L),
    \end{aligned}
\end{equation*}
where the first isomorphism is the CoYoneda lemma and the final isomorphism follows from Lemma \ref{lem coyoneda for simplicial complexes}.
\end{proof}

The functor $\Sing$ has a left adjoint $\Re$ given by
\begin{equation*}
    \Re(X) = \int^n X_n \times \bDelta^n,
\end{equation*}
where the colimit is taken in $\ncat{Cpx}$. We call this the \textbf{realization functor}. We can equivalently describe this functor as the left Kan extension
\begin{equation*}
    \begin{tikzcd}
	{\mathsf{\Delta}} & {\ncat{Cpx}} \\
	{\ncat{sSet}}
	\arrow["{\bDelta^\bullet}", from=1-1, to=1-2]
	\arrow["y"', hook', from=1-1, to=2-1]
	\arrow["{\Re}"', from=2-1, to=1-2]
\end{tikzcd}
\end{equation*}
where $y : \mathsf{\Delta} \to \ncat{sSet}$ is the Yoneda embedding and $\bDelta^\bullet : \mathsf{\Delta} \to \ncat{Cpx}$ is the cosimplicial simplicial complex $[n] \mapsto \bDelta^n$. Since $y$ is fully faithful, this implies that $\Re(\Delta^n) \cong \bDelta^n$.  

\begin{Prop} \label{prop explicit description of Re}
Given a simplicial set $X$, $\Re(X)$ is the simplicial complex with $V(\Re(X)) = X_0$ and $\{x_0, \dots, x_n \} \in \Re(X)$ if and only if there exists an $n$-simplex $\sigma \in X_n$ with $\sigma(i) = x_i$ for all $0 \leq i \leq n$, where $\sigma(i)$ is the $i$th vertex of $\sigma$.
\end{Prop}

\begin{proof}
First note that every simplicial set $X$ can be written as the union of its skeletons $X = \cup_{n \geq 0} \text{sk}_n(X)$. Each skeleton is given by a pushout along the previous skeleton
\begin{equation*}
    \begin{tikzcd}
	{\sum_{\sigma \in X^\text{nd}_n} \partial \Delta^n} & {\text{sk}_{n-1}(X)} \\
	{\sum_{\sigma \in X^{\text{nd}}_n} \Delta^n} & {\text{sk}_n(X)}
	\arrow[from=1-1, to=1-2]
	\arrow[hook, from=1-1, to=2-1]
	\arrow[hook, from=1-2, to=2-2]
	\arrow[from=2-1, to=2-2]
	\arrow["\lrcorner"{anchor=center, pos=0.125, rotate=180}, draw=none, from=2-2, to=1-1]
\end{tikzcd}
\end{equation*}
see \cite[Section 20]{rezk2022introduction} for more details.

We know that $\Re$ is a left adjoint, and so we obtain the pushout
\begin{equation*}
    \begin{tikzcd}
	{\sum_{\sigma \in X^\text{nd}_n} \partial \u{\Delta^n}} & {\Re(\text{sk}_{n-1}(X))} \\
	{\sum_{\sigma \in X^{\text{nd}}_n} \bDelta^n} & {\Re(\text{sk}_n(X))}
	\arrow[from=1-1, to=1-2]
	\arrow[hook, from=1-1, to=2-1]
	\arrow[hook, from=1-2, to=2-2]
	\arrow[from=2-1, to=2-2]
	\arrow["\lrcorner"{anchor=center, pos=0.125, rotate=180}, draw=none, from=2-2, to=1-1]
\end{tikzcd}
\end{equation*}
in $\ncat{Cpx}$, where $\partial\u{\Delta}^n$ is the simplicial complex obtained\footnote{We can see that $\Re(\partial \Delta^n) \cong \partial \u{\Delta}^n$ by using the characterization of $\partial \Delta^n$ as a coequalizer of simplices, see \cite[Section I.3]{goerss2009simplicial}.} by deleting the unique facet in $\bDelta^n$ and $X^{\text{nd}}_n$ is the set of nondegenerate $n$-simplices of $X$. Note that if $\sigma \in X^{\text{nd}}_n$ has any repeated vertices, then the map $\partial \u{\Delta}^n \to \Re(\text{sk}_{n-1}(X))$ corresponding to $\sigma$ will collapse down to a map $\tau : \bDelta^m \to \Re(\text{sk}_{n-1}(X))$ for $m < n$. But this is only possible if there is already a simplex in $\Re(\text{sk}_{n-1}(X))$ with the same vertices as the image of $\tau$. In other words, if $\sigma \in X^{\text{nd}}_n$ has any repeated vertices, then there must exist some other $\sigma' \in \text{sk}_{n-1}(X)^\text{nd}$ such that $\Re(\sigma) = \Re(\sigma')$. Hence for simplices with repeated vertices, the pushout adds nothing to the $n$-skeleton. Thus we can replace all of the above pushouts with
\begin{equation*}
    \begin{tikzcd}
	{\sum_{\sigma \in NSX^\text{nd}_n} \partial \u{\Delta}^n} & {\Re(\text{sk}_{n-1}(X))} \\
	{\sum_{\sigma \in NSX^{\text{nd}}_n} \bDelta^n} & {\Re(\text{sk}_n(X))}
	\arrow[from=1-1, to=1-2]
	\arrow[hook, from=1-1, to=2-1]
	\arrow[hook, from=1-2, to=2-2]
	\arrow[from=2-1, to=2-2]
	\arrow["\lrcorner"{anchor=center, pos=0.125, rotate=180}, draw=none, from=2-2, to=1-1]
\end{tikzcd}
\end{equation*}
where $NSX \subseteq X$ is the sub-simplicial set consisting of nonsingular simplices, i.e. those simplices with distinct vertices.

Now if $\{x_0, \dots, x_n \}$ is a collection of vertices, then the only way that it can be filled in with a simplex is if there is a nondegenerate, nonsingular simplex $\sigma \in NSX^{\text{nd}}_n$ whose boundary is already filled in. Using a simple induction argument, the result follows.
\end{proof}

\begin{Rem}
Note that $\Re$ only ``sees'' the nondegenerate simplices of $X$, as degenerate simplices always have repeated vertices. Furthermore, given a map $f : X \to Y$ of simplicial sets, $\Re(f) : \Re(X) \to \Re(Y)$ is well-defined, because if $x \in X_n$ is nonsingular and nondegenerate, but $f(x)$ is degenerate, then there exists a unique (possibly singular) non-degenerate simplex $y \in Y$ such that $f(x) = \sigma(y)$, by the Eilenberg-Zilber Lemma \cite[Proposition 19.14]{rezk2022introduction} where $\sigma$ is a composition of degeneracy operators. Furthermore if $y$ is singular, then $\Re(y)$ is identified with the image $\Re(z)$ of some lower-dimensional nonsingular, nondegenerate simplex $z$ and $\Re(f)$ maps $x \in \Re(X)$ to $z$.
\end{Rem}

\begin{Rem}
Note that $\Re$ can drastically change the homotopy type of a simplicial set. For instance, consider the simplicial set $X = \Delta^1/\partial \Delta^1$, which has one vertex and one nondegenerate 1-simplex. Then $|X| \simeq S^1$, but $\Re(X) \cong \bDelta^0$, and hence $|\Re(X)| \cong *$ where $*$ denotes a point.
\end{Rem}

\begin{Rem}
Sadly, $\Re$ does not preserve finite products. Indeed $\Re(\Delta^1 \times \Delta^1)$ is a simplicially subdivided square, and $\Re(\Delta^1) \times \Re(\Delta^1) \cong \bDelta^1 \times \bDelta^1 \cong \bDelta^3$.
\end{Rem}

\begin{Lemma}
Given an ordered simplicial complex $K_{\leq}$, with underlying simplicial complex $K = U(K_\leq)$, there is an isomorphism
\begin{equation} \label{eq realization on simp complexes}
\Re(\oSing(K_{\leq})) \xrightarrow{\cong} K.
\end{equation}
Furthermore for any simplicial complex $K$, the counit
\begin{equation*}
\varepsilon_K : \Re(\Sing(K)) \xrightarrow{\cong} K
\end{equation*}
is an isomorphism.
\end{Lemma}

\begin{proof}
The first claim is immediate from Proposition \ref{prop explicit description of Re} and the second is by Lemma \ref{lem sing is fully faithful}.
\end{proof}

\begin{Lemma}[{\cite[Proposition 3.52]{ruschoff2017lecture}}] \label{lem realization preserves subdivision}
Given an ordered simplicial complex $K_{\leq}$, with underlying simplicial complex $K = U(K_\leq)$, there is a natural isomorphism
\begin{equation*}
\sd \, \oSing(K_{\leq}) \cong \oSing( \Sd \, K),
\end{equation*}
where $\Sd \, K$ is considered as an ordered simplicial complex as in Definition \ref{def subdivision}.
\end{Lemma}

\begin{Cor} \label{cor realization of subdivision of horns, simplices, boundaries}
Given any ordered simplicial complex $K_{\leq}$, there is an isomorphism
\begin{equation*}
    \Re(\sd \,  \oSing(K_\leq)) \xrightarrow{\cong} \Sd K.
\end{equation*}
In particular, if we let $\u{\Lambda}^n_k = \Re( \Lambda^n_k)$, then we have
\begin{equation*}
    \Re( \sd \, \Lambda^n_k) \cong \Sd \, \u{\Lambda}^n_k, \qquad \Re(\sd \, \Delta^n) \cong \Sd \, \bDelta^n, \qquad \Re(\sd \, \partial \Delta^n) \cong \Sd \, \partial \bDelta^n,
\end{equation*}
for $n \geq 0$ and $1 \leq k \leq n$.
\end{Cor}

The following result shows that $\Sing$ behaves correctly with respect to geometric realization.

\begin{Prop}[{\cite{camarena2019turning}}] \label{prop hoequiv of realizations}
Given a simplicial complex $K$, there is a natural homotopy equivalence
\begin{equation*}
    |K|_{\ncat{Cpx}} \xrightarrow{\simeq} |\Sing(K)|.
\end{equation*}
\end{Prop}

\begin{Rem}
See \cite{ramras2022simplicial} for an alternate proof of the above result.
\end{Rem}

In fact, Antol\'{i}n-Camarena in \cite{camarena2019turning} proves a stronger statement than the above Proposition, given by the following result.

\begin{Lemma} \label{lem forgetting ordering on sing is weak equiv}
Given an ordered simplicial complex $K_\leq$ with underlying simplicial complex $K = U(K_{\leq})$, the map
\begin{equation*}
   \varphi_{K_\leq}:  \oSing(K_\leq) \to \Sing(K)
\end{equation*}
natural in $K_\leq$, defined on an oriented simplex $\bDelta^n \to K_\leq$ by just forgetting the orientation, is a weak equivalence of simplicial sets.
\end{Lemma}

\begin{proof}
The map $\varphi_{K_{\leq}}$ above is precisely the map \begin{equation*}
    \oSing(K_\leq ) \cong \Delta \otimes_{\bDelta} \oSing(K_\leq ) \to E \otimes_{\bDelta} \oSing(K_\leq ) \cong \Sing(K)
\end{equation*}
in \cite{camarena2019turning}, which Antol\'{i}n-Camarena proves is a weak equivalence of simplicial sets.
\end{proof}

\begin{Lemma} \label{lem osing of last vertex map is weak equiv}
Given an ordered simplicial complex $K_\leq$, the map
\begin{equation*}
    \sd \, \oSing (K_\leq) \cong \oSing( \Sd K_\leq) \xrightarrow{\oSing(\lambda)} \oSing (K_\leq),
\end{equation*}
where $\lambda : \Sd K_\leq \to K_\leq$ is the last vertex map is a weak equivalence of simplicial sets.
\end{Lemma}

\begin{proof}
It is easy to see that $\oSing(\lambda)$ is precisely the last vertex map of simplicial sets. Hence by Proposition 3.3.4.8 of \cite[\href{https://kerodon.net/tag/00YR}{Tag 00YR}]{kerodon} $\oSing(\lambda)$ is a weak equivalence of simplicial sets.
\end{proof}

\begin{Cor}
Given an ordered simplicial complex $K_\leq$ with underlying simplicial complex $K = U(K_\leq)$, the map
\begin{equation}
    \Sing(\Sd K) \xrightarrow{\Sing(\lambda)} \Sing(K)
\end{equation}
where $\lambda$ is the last vertex map (\ref{eq last vertex map}), is a weak equivalence of simplicial sets.
\end{Cor}

\begin{proof}
Since the map $\varphi$ in Lemma \ref{lem forgetting ordering on sing is weak equiv} is natural in $K_{\leq}$ we get a commutative diagram
\begin{equation*}
\begin{tikzcd}
	{\oSing(\Sd K_\leq)} & {\oSing(K_\leq)} \\
	{\Sing(\Sd K)} & {\Sing(K)}
	\arrow["{\oSing(\lambda)}", from=1-1, to=1-2]
	\arrow["{\varphi_{\Sd K_\leq}}"', from=1-1, to=2-1]
	\arrow["{\varphi_{K_\leq}}", from=1-2, to=2-2]
	\arrow["{\Sing(\lambda)}"', from=2-1, to=2-2]
\end{tikzcd}
\end{equation*}
in which the two vertical maps are weak equivalences by Lemma \ref{lem forgetting ordering on sing is weak equiv}. The top horizontal map is a weak equivalence by Lemma \ref{lem osing of last vertex map is weak equiv}. Hence by $2$-of-$3$, the bottom horizontal map is a weak equivalence.
\end{proof}

\subsection{Notions of Homotopy for Simplicial Complexes}
Let us now introduce some notions of homotopy for simplicial complexes.

\begin{Def} \label{def homotopy of simplicial complexes}
We say that two maps $f, g: K \to L$ of simplicial complexes are 
\begin{enumerate}
    \item \textbf{contiguous} if for all $\sigma \in K$, $f(\sigma) \cup g(\sigma) \in L$. We write $f \sim_c g$ to mean that $f$ and $g$ are contiguous,
    \item \textbf{$n$-homotopic} for $n \geq 1$, if there is a map
\begin{equation*}
    H : K \times I_n \to L
\end{equation*}
such that $H(-,0) = f$ and $H(-,n) = g$, which we call an $n$-homotopy, where $I_n$ is the \textbf{$n$-path}, i.e. the simplicial complex with $n+1$ vertices and $V(I_n) = \{0, 1, \dots, n \}$ and edges between $i$ and $i+1$ for $0 \leq i \leq n - 1$. We denote the $n$-homotopic relation by $f \simeq_n g$. 
\item \textbf{$\times$-homotopic} if $f$ and $g$ are $n$-homotopic for some $n \geq 1$. We denote the $\times$-homotopic relation by $f \simeq_{\times} g$.
\end{enumerate}
We say that a map $f : K \to L$ is an $\times$-homotopy equivalence if there exists a map $g : L \to K$ and $\times$-homotopies $gf \simeq_\alpha 1_K$ and $fg \simeq_\alpha 1_L$.
\end{Def}

\begin{Rem}
Note that the contiguity relation $\simeq_c$ is symmetric and reflexive but not transitive. 
\end{Rem}

\begin{Lemma}
The $\times$-homotopy relation is an equivalence relation.
\end{Lemma}

\begin{proof}
Clearly $\simeq_\times$ is reflexive. Now suppose that $f \simeq_\times g$ is witnessed by the homotopy $H : K \times I_n \to L$. Then the composite map
\begin{equation*}
    K \times I_n \xrightarrow{1_K \times \tau} K \times I_n \xrightarrow{H} L,
\end{equation*}
where $\tau : I_n \to I_n$ is defined by $\tau(k) = n-k$ for $0 \leq k \leq n$, witnesses the homotopy $g \simeq_\times f$. Thus $\simeq_\times$ is symmetric.

Now suppose that $f \simeq_\times g$ and $g \simeq_\times h$ witnessed by the homotopies $H : K \times I_n \to L$ and $H' : K \times I_m \to L$ respectively. Define the map $H'' : K \times I_{n + m} \to L$ by 
\begin{equation*}
    H''(k, i) = \begin{cases}
        H(k, i), & \text{if } 0 \leq i \leq n \\
        H'(k, i-n), & \text{if } n \leq i \leq m.
    \end{cases}
\end{equation*}
This gives an $\times$-homotopy $f \simeq_\times h$.
\end{proof}

\begin{Lemma}
If $f, f': K \to L$ and $g, g' : L \to M$ are maps of simplicial complexes, $f \simeq_\times f'$ and $g \simeq_\times g'$, then $gf \simeq_\times g'f'$.
\end{Lemma}

\begin{proof}
Suppose we have homotopies $H : K \times I_n \to L$ and $H' : L \times I_{n'} \to M$ witnessing $f \simeq_\times f'$ and $g \simeq_\times g'$ respectively. Let $N = \max \{n, n' \}$, then define a homotopy $H'' : K \times I_N \to M$ by
\begin{equation} \label{eq composite homotopy formula}
    H''(k, i) = H'(H(k, \min(i, n)), \min(i, n')).
\end{equation}
It is easy to see this defines a map of simplicial complexes, and $H''(k, 0) = H'(H(k,0),0) = gf(k)$ and $H''(k, N) = H'(H(k,n), n') = H'(g(k), n') = g'f'(k)$. Thus $H''$ defines a $\times$-homotopy $H'' : gf \simeq_\times g'f'$.
\end{proof}

\begin{Lemma}[{\cite[Corollary A.1.3]{barmak2011algebraic}}] \label{lem realization of contiguous maps are homotopic}
If $f, g: K \to L$ are maps of simplicial complexes such that $f \simeq_c g$, then $|f| \simeq |g|$ as maps of topological spaces.
\end{Lemma}

\begin{Lemma} \label{lem 1-homotopies are contiguities}
If $f, g: K \to L$ are maps of simplicial complexes such that $f \simeq_1 g$, then $f \simeq_c g$.
\end{Lemma}

\begin{proof}
Suppose that there is a map $H : K \times \Delta^1 \to L$ of simplicial complexes such that $H(-,0) = f$ and $H(-,1) = g$. If $\sigma = \{x_0, \dots, x_n \} \in K$, then $\sigma' = \{(x_0, 0), \dots, (x_n, 0), (x_0, 1), \dots, (x_n, 1) \} \in K \times \Delta^1$. Since $H$ is a map of simplicial complexes, $H(\sigma') \in L$, but $H(\sigma') = f(\sigma) \cup g(\sigma)$. Hence $f \simeq_c g$.
\end{proof}

Now suppose that $n \geq 1$ and $f, g: K \to L$ are maps of simplicial complexes such that $f \simeq_n g$. Then we have an $n$-homotopy $H : K \times I_n \to L$. For each $0 \leq m \leq n$, we obtain a map $f_m : K \to L$ defined as the composite map
\begin{equation*}
   K \xrightarrow{i_m} K \times I_n \xrightarrow{H} L 
\end{equation*}
where $i_m(x) = (x, m)$. We therefore obtain a sequence of $1$-homotopies
\begin{equation} \label{eq sequence of 1-homotopies}
   f = f_0 \simeq_1 f_1 \simeq_1 f_2 \simeq_1 \dots \simeq_1 f_n = g. 
\end{equation}

\begin{Cor} \label{cor realization of x-homotopic is homotopic}
If $f, g: K \to L$ are maps of simplicial complexes such that $f \simeq_\times g$, then $|f|_{\ncat{Cpx}} \simeq |g|_{\ncat{Cpx}}$ as maps of topological spaces.
\end{Cor}

\begin{proof}
Since $f \simeq_\times g$, we obtain a sequence of $1$-homotopies as in (\ref{eq sequence of 1-homotopies}), but by Lemma \ref{lem 1-homotopies are contiguities}, this gives a sequence $f_0 \simeq_c f_1 \simeq_c \dots \simeq_c f_n$, which by Lemma \ref{lem realization of contiguous maps are homotopic} provides a sequence $|f_0|_{\ncat{Cpx}} \simeq |f_1|_{\ncat{Cpx}} \simeq \dots \simeq |f_n|_{\ncat{Cpx}}$ of homotopies of maps of topological spaces. But the homotopy relation on spaces is transitive, hence $|f|_{\ncat{Cpx}} = |f_0|_{\ncat{Cpx}} \simeq |f_n|_{\ncat{Cpx}} = |g|_{\ncat{Cpx}}$.
\end{proof}

\begin{Cor} \label{cor realization of x-hoequiv is hoequiv}
If $f : K \to L$ is an $\times$-homotopy equivalence of simplicial complexes, then $|f| : |K|_\ncat{Cpx} \to |L|_\ncat{Cpx}$ is a homotopy equivalence of topological spaces.
\end{Cor}

\begin{Lemma}
If $f, g: K \to L$ are maps of simplicial complexes and $f \simeq_1 g$, then $\Sing(f) \simeq \Sing(g)$ as maps of simplicial sets.
\end{Lemma}

\begin{proof}
Given a $1$-homotopy $H : K \times \Delta^1 \to L$, we obtain a map
\begin{equation} \label{eq sing of homotopy gives joyal homotopy}
    \Sing(H) : \Sing(K \times \Delta^1) \cong \Sing(K) \times \mathbb{J} \to \Sing(L)
\end{equation}
such that $\Sing(H)(-,0) = \Sing(f)$ and $\Sing(H)(-,1) = \Sing(g)$, where $\mathbb{J} = N(E[1])$ is the nerve of the category with two distinct objects and exactly one morphism between every pair of objects. There is an inclusion $\Delta^1 \hookrightarrow \mathbb{J}$ of simplicial sets, from which we obtain a simplicial homotopy $H' : \Sing(K) \times \Delta^1 \to \Sing(L)$.
\end{proof}

\begin{Rem}
We note that $\mathbb{J}$ is an interval object for the Joyal model structure on $\ncat{sSet}$, whose fibrant objects are quasicategories. Hence the map (\ref{eq sing of homotopy gives joyal homotopy}) is an example of a homotopy in the Joyal model structure. See \cite[Section 15.3]{riehl2014categorical} for a nice discussion about this notion of homotopy.
\end{Rem}

\begin{Prop} \label{prop Sing preserves x-homotopy equivalence}
Suppose that $f : K \to L$ is a $\times$-homotopy equivalence of simplicial complexes. Then $\Sing(f) : \Sing(K) \to \Sing(L)$ is a weak equivalence of simplicial sets.
\end{Prop}

\begin{proof}
By Propostion \ref{prop hoequiv of realizations} we obtain a commutative diagram
\begin{equation*}
    \begin{tikzcd}
	{|K|_{\ncat{Cpx}}} & {|L|_{\ncat{Cpx}}} \\
	{|\Sing(K)|} & {|\Sing(L)|}
	\arrow["{|f|_{\ncat{Cpx}}}", from=1-1, to=1-2]
	\arrow["\simeq"', from=1-1, to=2-1]
	\arrow["\simeq", from=1-2, to=2-2]
	\arrow["{|\Sing(f)|}"', from=2-1, to=2-2]
\end{tikzcd}
\end{equation*}
where the vertical maps are homotopy equivalences. By Corollary \ref{cor realization of x-hoequiv is hoequiv}, the top horizontal map is a homotopy equivalence. Hence by $2$-of-$3$, $|\Sing(f)|$ is a homotopy equivalence. Hence $\Sing(f)$ is a weak equivalence of simplicial sets.
\end{proof}

\begin{Def}
Given a pair $(L, K)$ of simplicial complexes with inclusion map $i : K \hookrightarrow L$, we say that $K$ is an \textbf{$\times$-deformation retract} of $L$ if there exists a map $r : L \to K$ such that $ri = 1_K$ and an $\times$-homotopy $H: L \times I_n \to L$ such that $H(x,m) = x$ for all $x \in V(K)$ and all $0 \leq m \leq n$\footnote{We should really call this a $\times$-deformation retract rel $K$ since we require the $\times$-homotopy is constant on $K$, but we believe that this abuse of terminology will not cause confusion.}. We say that $i : K \hookrightarrow L$ is the inclusion of a $\times$-deformation retract.
\end{Def}

\begin{Lemma} \label{lem composition of deformation retracts}
Suppose that $i : K \hookrightarrow L$ and $j : L \hookrightarrow M$ are inclusions of $\times$-deformation retracts. Then the pair $ji : K \hookrightarrow M$ is the inclusion of an $\times$-deformation retract.
\end{Lemma}

\begin{proof}
Let $r : L \to K$, $s : M \to L$ denote the corresponding retraction maps. There are homotopies $H : L \times I_n \to L$ and $H' : M \times I_{n'} \to M$ such that $H(-,0) = 1_L$, $H(-,n) = ir$, $H'(-,0) = 1_M$, $H'(-,n') = js$ and if $k \in V(K)$ then $H(k,i) = k$ and if $\ell \in V(L)$, then $H'(\ell, i) = \ell$ for all $i$. Let $N = n + n'$, and define a homotopy $H'' : M \times I_n \to M$ by
\begin{equation*}
    H''(m, i) = \begin{cases}
        H'(m, i), & \text{if } 0 \leq i \leq n', \\
        H(s(m), i - n'), & \text{if } n' \leq i \leq N.
    \end{cases}
\end{equation*}
This is clearly a map of simplicial complexes. We have $H''(m, 0) = H'(m,0) = m$ and $H''(m, N) = H(s(m), n) = r(s(m))$. If $k \in V(K)$, then $H''(k, i)$ is either $H'(k, i) = k$ since $k \in V(L)$, or $H(s(k), i - n') = H(k, i - n') = k$ since $s$ is the identity on $L$. Thus $H''$ defines a homotopy $1_M \simeq_\times ji rs$, and is hence a $\times$-deformation retract of $M$ onto $K$.
\end{proof}

The following notions will be useful for constructing $\times$-deformation retracts by hand.

\begin{Def}
Given a simplicial complex $K$, we say that a vertex $v \in V(K)$ is \textbf{dominated} by another vertex $v'$ if and only if all of the facets of $K$ that contain $v$ also contain $v'$.

If $\sigma \in K$, then we let $K \setminus \sigma$ denote the simplicial complex defined by
\begin{equation*}
    K \setminus \sigma = \{ \tau \in K \, | \, \tau \cap \sigma = \varnothing \}.
\end{equation*}
We call $K \setminus \sigma$ the \textbf{deletion} of $\sigma$ from $K$.
An \textbf{elementary strong collapse} of a simplicial complex $K$ is the deletion $K \setminus v$ of a dominated vertex $v$ from $K$. We say that there is a \textbf{strong collapse} of a simplicial complex $K$ to a subcomplex $L$ if there is a finite sequence of elementary strong collapses from $K$ that result in $L$.
\end{Def}

\begin{Ex} \label{ex strong collapse}
In a simplex $\bDelta^n$, every vertex is dominated by every other vertex. Hence $n$-simplices strongly collapse to a point. Similarly, every vertex of $\u{\Lambda}^n_k$ is dominated by the vertex $k \in V(\u{\Lambda}^n_k)$, hence $\u{\Lambda}^n_k$ strongly collapses to a point. The boundary $\partial \u{\Delta}^n$ does not strongly collapse onto any of its subcomplexes.
\end{Ex}

Given a simplicial complex $K$, if $v \in V(K)$ is a vertex dominated by another vertex $v'$, then the function $V(r) : V(K) \to V(K) \setminus v$ defined by
\begin{equation} \label{eq strong collapse retraction map}
    V(r)(w) = \begin{cases}
        w & \text{if } w \neq v, \\
        v' & \text{if } w = v,
    \end{cases}
\end{equation}
defines a retraction map of simplicial complexes $r : K \to K \setminus v$.

\begin{Lemma} \label{lem strong collapse is x-def retract}
If $(L, K)$ is a pair of simplicial complexes such that $L$ strongly collapses to $K$, then $K$ is a $\times$-deformation retract of $L$.
\end{Lemma}

\begin{proof}
First let us show that if $K'$ is obtained from $K$ by an elementary strong collapse, then $K'$ and $K$ are $\times$-homotopy equivalent. So suppose that $v \in V(K)$ is dominated by a vertex $v' \in V(K')$ such that $K' = K \setminus v$. Let $i : K' \to K$ be the inclusion map, and let $r : K \to K'$ be the retraction map (\ref{eq strong collapse retraction map}). Let us define a $1$-homotopy $H : K \times I_1 \to K$ by setting
\begin{equation*}
    H(x,-) = x \text{ for } x \in V(K'), \qquad H(v,i) = \begin{cases}
        v, & \text{if } i = 0, \\
        v', & \text{if } i = 1.
    \end{cases}
\end{equation*}
To show that this a map of simplicial complexes, we need to show that every simplex of the form $A = \{(x_0, 0), \dots, (x_m, 0), (y_0, 1), \dots, (y_\ell, 1) \}$ where $B = \{x_0, \dots, x_m, y_0, \dots, y_\ell \} \in K$ is sent to a simplex in $K'$. But this is clear because if $v \notin B$, then $H(A) = B$, and if $v \in B$, then $(B \setminus v) \cup v'$ is a simplex of $K'$ since all the facets of $K$ that contain $v$ also contain $v'$.
Clearly $H(-,0) = 1_K$, $H(-,1) = i r$ and $H(x,i) = x$ for all $x \in K'$, $i = 0, 1$. Thus $H$ is a $\times$-deformation retract. Composites of $\times$-deformation retracts are $\times$-deformation retracts by Lemma \ref{lem composition of deformation retracts}, hence the result follows.
\end{proof}

\begin{Lemma}[{\cite[Corollary 4.11]{matsushita2017box}}] \label{lem sd preserves strong collapse}
Suppose that $(L, K)$ is a pair of simplicial complexes and $L$ strongly collapses to $K$. Then $\Sd \, L$ strongly collapses to $\Sd \, K$.
\end{Lemma}

\begin{Def}
Given a simplicial complex $K$ and a simplex $\sigma \in K$, the \textbf{closed star} of $\sigma$ is the subcomplex $\st_K(\sigma) \subseteq K$ defined by
\begin{equation*}
    \st_K(\sigma) = \{ \tau \in K \, | \, \sigma \cup \tau \in K \}.
\end{equation*}
Given a simplicial complex $L$ and a subcomplex $K \subseteq L$, the \textbf{neighborhood} of $K$ in $L$, written $n_L(K)$ is the subcomplex
\begin{equation*}
    n_L(K) = \bigcup_{v \in V(K)} \st_L(v).
\end{equation*}
We say that the pair $(L, K)$ is a \textbf{strong NDR pair} if there exists a subcomplex $L' \subseteq L$ such that $n_L(K) \subseteq L'$, and such that $L'$ strong collapses to $K$.

We say that $(M, K)$ is an \textbf{$\times$-NDR pair} if there exists a subcomplex $L \subseteq M$ such that $n_M(K) \subseteq L$ and such that $K$ is an $\times$-deformation retract of $L$. We say that the inclusion map $i : K \hookrightarrow M$ is the \textbf{inclusion of an $\times$-NDR pair}, or we may just call it a $\times$-deformation retract.
\end{Def}

\begin{Rem}
Note that if $K \subseteq L$ is a subcomplex and $\sigma \in n_L(K)$, then there must exist a $v \in K$ such that $\sigma \in \st_L(v)$. Also note that if $(L, K)$ is a strong NDR pair, then it is also an $\times$-NDR pair by Lemma \ref{lem strong collapse is x-def retract}.
\end{Rem}

\begin{Lemma} \label{lem pushout of deformation retract}
Suppose that
\begin{equation*}
    \begin{tikzcd}
	K & X \\
	L & Y
	\arrow["u", from=1-1, to=1-2]
	\arrow["i"', hook, from=1-1, to=2-1]
	\arrow["j", from=1-2, to=2-2]
	\arrow["v"', from=2-1, to=2-2]
	\arrow["\lrcorner"{anchor=center, pos=0.125, rotate=180}, draw=none, from=2-2, to=1-1]
\end{tikzcd}
\end{equation*}
is a pushout square in $\ncat{Cpx}$ where $i$ is the inclusion of an $\times$-deformation retract. Then $j$ is the inclusion a $\times$-deformation retract.
\end{Lemma}

\begin{proof}
By Corollary \ref{cor monos and epis of simplicial complexes}, a map $f : K \to L$ is a monomorphism if and only if $V(f) : V(K) \to V(L)$ is a monomorphism of sets. Now the forgetful functor $V : \ncat{Cpx} \to \ncat{Set}$ preserves colimits because it is a left adjoint. Pushouts of monomorphisms in $\ncat{Set}$ are monomorphisms, hence pushouts of monomorphisms in $\ncat{Cpx}$ are monomorphisms. Thus $j$ is a monomorphism of simplicial complexes. We now want to show that $j$ is an $\times$-deformation retract. We have the following solid commutative diagram
\begin{equation*}
    \begin{tikzcd}
	K & X \\
	L & Y \\
	K & X
	\arrow["u", from=1-1, to=1-2]
	\arrow["i", hook, from=1-1, to=2-1]
	\arrow["{1_K}"', curve={height=18pt}, from=1-1, to=3-1]
	\arrow["j"', hook, from=1-2, to=2-2]
	\arrow["{1_X}", curve={height=-18pt}, from=1-2, to=3-2]
	\arrow["v", from=2-1, to=2-2]
	\arrow["r", from=2-1, to=3-1]
	\arrow["\lrcorner"{anchor=center, pos=0.125, rotate=180}, draw=none, from=2-2, to=1-1]
	\arrow["s"', dashed, from=2-2, to=3-2]
	\arrow["u"', from=3-1, to=3-2]
\end{tikzcd}
\end{equation*}
and by the universal property of the pushout, we obtain the unique dotted map $s$ making the diagram commute. By the pushout pasting Lemma \ref{lem pushout pasting lemma}, $s$ is also a pushout, since the top square is a pushout and the outer rectangle is a pushout.

Now let $H : L \times I_n \to L$ denote an $\times$-homotopy witnessing $K$ as an $\times$-deformation retract of $L$, i.e. $H : ir \simeq_n 1_L$. Since $\ncat{Cpx}$ is Cartesian closed, $(-) \times I_n$ preserves colimits. Hence
\begin{equation*}
    \begin{tikzcd}
	{K \times I_n} & {X \times I_n} \\
	{L \times I_n} & {Y \times I_n}
	\arrow["{u \times 1_{I_n}}", from=1-1, to=1-2]
	\arrow["{i \times 1_{I_n}}"', hook, from=1-1, to=2-1]
	\arrow["{j \times 1_{I_n}}", hook, from=1-2, to=2-2]
	\arrow["{v \times 1_{I_n}}"', from=2-1, to=2-2]
	\arrow["\lrcorner"{anchor=center, pos=0.125, rotate=180}, draw=none, from=2-2, to=1-1]
\end{tikzcd}
\end{equation*}
is a pushout square. We have maps $v H : L \times I_n \to Y$ and $j \, \text{proj}_X : X \times I_n \to Y$ such that for $k \in V(K)$ and $m \in V(I_n)$ we have
\begin{equation*}
vH (i \times 1_{I_n})(k,m) = vH(k,m) = vi(k) = ju(k) = j \text{proj}_X (u \times 1_{I_n})(k,m).
\end{equation*}
hence by the universal property of the pushout we obtain a unique map $H' : Y \times I_n \to Y$ making the following diagram commute
\begin{equation*}
 \begin{tikzcd}
	{K \times I_n} & {X \times I_n} \\
	{L \times I_n} & {Y \times I_n} \\
	&& Y
	\arrow["{{u \times 1_{I_n}}}", from=1-1, to=1-2]
	\arrow["{{i \times 1_{I_n}}}"', hook, from=1-1, to=2-1]
	\arrow["{{j \times 1_{I_n}}}", hook, from=1-2, to=2-2]
	\arrow["{j \text{proj}_X}", curve={height=-12pt}, from=1-2, to=3-3]
	\arrow["{{v \times 1_{I_n}}}"', from=2-1, to=2-2]
	\arrow["vH"', curve={height=12pt}, from=2-1, to=3-3]
	\arrow["\lrcorner"{anchor=center, pos=0.125, rotate=180}, draw=none, from=2-2, to=1-1]
	\arrow["{H'}"', from=2-2, to=3-3]
\end{tikzcd}   
\end{equation*}
Now since $Y$ is a pushout of $L$ and $X$ along $K$, every vertex of $Y$ either comes from $L$ or $X$ or both. If $x \in V(Y)$ comes from $L$, i.e. $x = v(x')$ for $x' \in V(L)$, then \begin{equation*}
    H'(x,0) = H'(v(x'),0) = vH(x',0)= vir(x') = jur(x') = jsv(x') = js(x),
\end{equation*} and $H'(x,1) = vH(x', 1) = v(x') = x$, and if $x \in V(Y)$ comes from $X$, then $H'(x,0) = x$ and $H'(x,1) = 1$. In the case where $x$ is in both $L$ and $X$, then $x$ belongs to $K$, and clearly $H'$ is well defined in that case. Hence $H'$ defines an $\times$-deformation retract of $Y$ to $X$.
\end{proof}

The following result is vital to proving Theorem \ref{th transferred model structure from sset}, it is the key geometric result we need to understand what kind of pushouts $\Sing$ sends to homotopy pushouts, see Proposition \ref{prop sing sends NDR pairs to homotopy pushouts}.

\begin{Prop}[{\cite[Theorem 4.17]{matsushita2017box}}] \label{prop double subdivision is NDR pair}
If $K \subseteq L$ is an inclusion of a subcomplex, then $(\Sd^2(L), \Sd^2(K))$ is a strong NDR pair. 
\end{Prop}

\section{Transferred Model Structure on Simplicial Complexes} \label{section transfer simplicial complexes}

Let us review cofibrantly generated model categories. 

\begin{Def} \label{def cofibrantly generated model category}
Given a category $\cat{C}$ with colimits and a set $I$ of morphisms of $\cat{C}$, we let $\text{cell}(I)$ denote the class of morphisms obtained by taking transfinite composites of pushouts of maps in $I$. We call a morphism in $\text{cell}(I)$ a \textbf{relative $I$-cell complex}. We say an object $X \in \cat{C}$ is an \textbf{$I$-cell complex} if the unique map $\varnothing \to X$ from the initial object is a relative $I$-cell complex.

Let $\text{cof}(I)$ denote the class of retracts of morphisms in $\text{cell}(I)$. We say that a model category $\cat{C}$ is \textbf{cofibrantly generated} if there exist sets of morphisms $I$ and $J$, called the sets of \textbf{generating cofibrations} and \textbf{generating trivial cofibrations} respectively, such that the class of cofibrations in $\cat{C}$ is precisely $\text{cof}(I)$, the class of trivial cofibrations is precisely $\text{cof}(J)$, and both $I$ and $J$ satisfy the small object argument. 

We say that a model category $\cat{C}$ is \textbf{left proper} if weak equivalences are preserved by pushout along cofibrations, \textbf{right proper} if weak equivalences are preserves by pullback along fibrations, and \textbf{proper} if it is both left and right proper. For details on these definitions, see \cite[Section 3.1]{balchin2021handbook} and \cite[Part 3]{rezk2022introduction}.
\end{Def}

Now let us consider right-transfer of model categories. Suppose that $\cat{C}$ is a model category and we have an adjunction
\begin{equation*}
\begin{tikzcd}
	{\cat{D}} && {\cat{C}}
	\arrow[""{name=0, anchor=center, inner sep=0}, "R"', shift right=3, from=1-1, to=1-3]
	\arrow[""{name=1, anchor=center, inner sep=0}, "L"', shift right=3, from=1-3, to=1-1]
	\arrow["\dashv"{anchor=center, rotate=-90}, draw=none, from=1, to=0]
\end{tikzcd}
\end{equation*}

\begin{Def} \label{def transferred weak equivs and fibrations}
We say that a map $f : X \to Y$ in $\cat{D}$ is a 
\begin{itemize}
    \item weak equivalence if $R(f)$ is a weak equivalence in $\cat{C}$,
    \item (trivial) fibration if $R(f)$ is a (trivial) fibration in $\cat{C}$, 
    \item cofibration if $f$ has the left lifting property with respect to all trivial fibrations, and anodyne if $f$ has the left lifting property with respect to all fibrations.
\end{itemize}
\end{Def}

When certain conditions hold, the above define a model structure on $\cat{D}$, called the \textbf{right-transferred model structure} on $\cat{D}$. The following result will be the main tool we will use to transfer model structures. For details and terminology see \cite[Section 7.3]{fiore2008model}.

\begin{Prop}[{\cite[Corollary 6.1]{fiore2010thomason}}] \label{prop right-transfer}
Suppose that $L : \cat{C} \rightleftarrows \cat{D} : R$ is an adjunction where $\cat{C}$ is a cofibrantly generated model category with generating cofibrations $I$ and generating trivial cofibrations $J$, and $\cat{D}$ has all small limits and colimits. Suppose further that
\begin{enumerate}
    \item for every morphism $i \in I$ and $j \in J$, the domains of $L(i)$ and $L(j)$ are small with respect to $\cat{D}$,
    \item for any ordinal $\lambda$ and any colimit preserving functor $X : \lambda \to \cat{C}$ (which we call a $\lambda$-sequence) such that $X_{\alpha} \to X_{\alpha + 1}$ is a weak equivalence for all $\alpha + 1 < \lambda$, the transfinite composite 
    $$X_0 \to \ncolim{\beta < \lambda} X_\beta$$ is a weak equivalence,
    \item For any ordinal $\lambda$ and any colimit preserving functor $Y : \lambda \to \cat{D}$, the functor $R$ preserves the colimit of $Y$, and
    \item if $j'$ is a pushout of $L(j)$ in $\cat{D}$ for $j \in J$, then $R(j')$ is a weak equivalence in $\cat{C}$.
\end{enumerate}
Then $\cat{D}$ has a cofibrantly generated model structure with generating cofibrations $L(I)$ and generating trivial cofibrations $L(J)$. A map $f$ in $\cat{D}$ is a weak equivalence (fibration) if and only if $R(f)$ is a weak equivalence (fibration) in $\cat{C}$.
\end{Prop}

Condition (2) of Proposition \ref{prop right-transfer} is made easier to check by the following result.

\begin{Lemma}[{\cite[Proposition 7.9]{fiore2008model}, \cite[Corollary 7.4.2]{hovey2007model}}] \label{lem transfinite composite of weak equivs}
Suppose that $\cat{C}$ is a cofibrantly generated model category such that the domains and codomains of the generating cofibrations and generating trivial cofibrations are finite. Then every transfinite composite of weak equivalences is a weak equivalence.
\end{Lemma}

We recall the Thomason model structure on $\ncat{Cat}$, the (large) category of small categories. Let $N : \ncat{Cat} \to \ncat{sSet}$ denote the nerve functor and $h : \ncat{sSet} \to \ncat{Cat}$ its left adjoint. 

\begin{Th}[{\cite[Theorem 4.9]{thomason1980cat}}]
The right-transferred model structure on $\ncat{Cat}$ transferred along the adjunction
\begin{equation}
\begin{tikzcd}
	{\ncat{Cat}} && {\ncat{sSet}}
	\arrow[""{name=0, anchor=center, inner sep=0}, "{{\text{Ex}^2 \, N}}"', shift right=3, from=1-1, to=1-3]
	\arrow[""{name=1, anchor=center, inner sep=0}, "{{h \, \sd^2}}"', shift right=3, from=1-3, to=1-1]
	\arrow["\dashv"{anchor=center, rotate=-90}, draw=none, from=1, to=0]
\end{tikzcd}
\end{equation}
from the Kan-Quillen model structure on $\ncat{sSet}$ exists. Furthermore, it is proper and cofibrantly generated.
\end{Th}

Similarly, we obtain an adjunction
\begin{equation} \label{eq thomason adjunction cpx and sset}
\begin{tikzcd}
	{\ncat{Cpx}} && {\ncat{sSet}}
	\arrow[""{name=0, anchor=center, inner sep=0}, "{{\tSing}}"', shift right=3, from=1-1, to=1-3]
	\arrow[""{name=1, anchor=center, inner sep=0}, "{\tRe}"', shift right=3, from=1-3, to=1-1]
	\arrow["\dashv"{anchor=center, rotate=-90}, draw=none, from=1, to=0]
\end{tikzcd}
\end{equation}
where we let
\begin{equation}
    \tRe = \Re \, \sd^2, \qquad \qquad \tSing = \text{Ex}^2 \Sing.
\end{equation}

Now let us transfer the Kan-Quillen model structure on $\ncat{sSet}$ to obtain an analogue of the Thomason model structure on $\ncat{Cpx}$.

Recall that the Kan-Quillen model structure on $\ncat{sSet}$ is cofibrantly generated with generating (trivial) cofibrations 
\begin{equation} \label{eq generating cofibs for sset}
    \begin{aligned}
        I & = \{\partial \Delta^n \hookrightarrow \Delta^n \, | \, n \geq 0 \} \\
        J & = \{\Lambda^n_k \hookrightarrow \Delta^n \, | \, n \geq 1, \, 0 \leq k \leq n \}
    \end{aligned}
\end{equation}

Hence we want to show that $\ncat{Cpx}$ obtains a right-transferred cofibrantly generated model structure with generating (trivial) cofibrations
\begin{equation} \label{eq gen cofibs for cpx}
\begin{aligned}
   \tRe(I) & = \{ \tRe \partial \Delta^n \to \tRe \Delta^n \, | \, n \geq 0 \} \cong \{ \Sd^2 \partial \u{\Delta}^n \hookrightarrow \Sd^2 \bDelta^n \, | \, n \geq 0 \}, \\
   \tRe(J) & = \{ \tRe \Lambda^n_k \to \tRe \Delta^n \, | \, n \geq 1, \, 0 \leq k \leq n \} \cong \{ \Sd^2 \u{\Lambda}^n_k \hookrightarrow \Sd^2 \bDelta^n \, | \, n \geq 1, \, 0 \leq k \leq n \}
\end{aligned}
\end{equation}
where the right-hand isomorphisms follow from Lemma \ref{lem realization preserves subdivision}.
    
Note that $\u{\Lambda}^n_k \cong \u{\Lambda}^n_0$ for every $0 \leq k \leq n$, since they are unordered simplicial complexes. Hence we can equivalently replace the above with
\begin{equation*}
    \tRe(J) = \{ \Sd^2  \u{\Lambda}^n_0 \hookrightarrow \Sd^2 \bDelta^n \, | \, n \geq 1 \}.
\end{equation*}

The first three conditions of Proposition \ref{prop right-transfer} are relatively easy to check. It is condition (4) that requires a good deal of work. The next two results are needed to establish (4).

\begin{Rem}
The statements and proofs of Proposition \ref{prop sing sends NDR pairs to homotopy pushouts} and Lemma \ref{lem unit maps are weak equivs} resemble but are not identical to \cite[Theorem 4.23]{matsushita2017box} and \cite[Lemma 3.10]{matsushita2017box} respectively.
\end{Rem}

\begin{Prop} \label{prop sing sends NDR pairs to homotopy pushouts}
Suppose that
\begin{equation} \label{eq original pushout square Sing proof}
    \begin{tikzcd}
	K & X \\
	L & Y
	\arrow["u", from=1-1, to=1-2]
	\arrow["i"', hook, from=1-1, to=2-1]
	\arrow["f", hook, from=1-2, to=2-2]
	\arrow["v"', from=2-1, to=2-2]
	\arrow["\lrcorner"{anchor=center, pos=0.125, rotate=180}, draw=none, from=2-2, to=1-1]
\end{tikzcd}
\end{equation}
is a pushout diagram in $\ncat{Cpx}$ where $i : K \to L$ is the inclusion of an $\times$-NDR pair of simplicial complexes. Then the commutative diagram
\begin{equation} \label{eq sing of original pushout square Sing proof}
    \begin{tikzcd}
	{\Sing(K)} & {\Sing(X)} \\
	{\Sing(L)} & {\Sing(Y)}
	\arrow["{\Sing(u)}", from=1-1, to=1-2]
	\arrow["{\Sing(i)}"', hook, from=1-1, to=2-1]
	\arrow["{\Sing(f)}", hook, from=1-2, to=2-2]
	\arrow["{\Sing(v)}"', from=2-1, to=2-2]
\end{tikzcd}
\end{equation}
is a homotopy pushout in the Kan-Quillen model structure on $\ncat{sSet}$, i.e. the canonical map
\begin{equation*}
    \Sing(X) +_{\Sing(K)} \Sing(L) \to \Sing(Y)
\end{equation*}
is a weak equivalence.
\end{Prop}

\begin{proof}
Since $(L, K)$ is an $\times$-NDR pair of simplicial complexes, there exists a subcomplex $L'$ of $L$ containing $n_L(K)$ such that $K$ is an $\times$-deformation retract of $L'$.

Consider the pushout
\begin{equation*}
\begin{tikzcd}
	K & X \\
	{L'} & {X'}
	\arrow["u", from=1-1, to=1-2]
	\arrow["{i'}"', hook, from=1-1, to=2-1]
	\arrow["f'", hook, from=1-2, to=2-2]
	\arrow["\ell"', from=2-1, to=2-2]
	\arrow["\lrcorner"{anchor=center, pos=0.125, rotate=180}, draw=none, from=2-2, to=1-1]
\end{tikzcd}    
\end{equation*}
By Lemma \ref{lem pushout of deformation retract}, $f' : X \to X'$ is the inclusion of an $\times$-deformation retract. Let us show that $n_Y(X) \subseteq X'$. First, let us identify $X$ and $X'$ with their images $f(X) \subseteq f'(X') \subseteq Y$. Suppose that $x \in V(X)$, then $\st_Y(x) = \{\sigma \in Y \, | \, \sigma \cup x \in Y \}$. Now suppose that $\sigma \in \st_Y(x)$. If $\sigma \cup x \in X$, then $\sigma \in X$, so $\sigma \in X'$. If $\sigma \cup x \notin X$, then since $\sigma \cup x \in Y$, and $Y \cong L +_K X$, then there exists a $\tau \in L$ such that $v(\tau) = \sigma \cup x$. But then there exists a $z \in \tau$ such that $v(z) = x$. By the construction of the pushout, Example \ref{ex pushouts along mono of simplicial complexes}, this implies there exists a $k \in V(K)$ such that $z = i(k)$ and $u(k) = x$. But then this implies that $\tau \cap K \neq \varnothing$, and $\tau$ is a simplex, so $\tau \in n_L(K)$ and hence $\tau \in L'$. Thus $\sigma \cup x \in X'$ and hence $\sigma \in X'$.

This gives us a commutative diagram consisting of the solid arrows in the following diagram
\begin{equation} \label{eq pasting pushouts Sing proof}
\begin{tikzcd}
	K & X \\
	{L'} & {X'} \\
	L & Y
	\arrow["u", from=1-1, to=1-2]
	\arrow["{i'}", hook, from=1-1, to=2-1]
	\arrow["i"', curve={height=18pt}, from=1-1, to=3-1]
	\arrow["f'"', from=1-2, to=2-2]
	\arrow["f", curve={height=-18pt}, from=1-2, to=3-2]
	\arrow["\ell", from=2-1, to=2-2]
	\arrow["{i''}", hook, from=2-1, to=3-1]
	\arrow["\lrcorner"{anchor=center, pos=0.125, rotate=180}, draw=none, from=2-2, to=1-1]
	\arrow["{f''}"', dashed, from=2-2, to=3-2]
	\arrow["v"', from=3-1, to=3-2]
\end{tikzcd}
\end{equation}
The universal property of the pushout provides us with a unique dotted map $f'' : X' \to Y$ such that the above diagram commutes. Since the top square is a pushout and the outer rectangle is a pushout, by the pushout pasting Lemma \ref{lem pushout pasting lemma}, the lower square is a pushout.

We claim that $\Sing$ sends the bottom pushout square to a pushout in $\ncat{sSet}$. Let us analyze the commutative square
\begin{equation} \label{eq strict sing pushout}
\begin{tikzcd}
	{\Sing(L')} & {\Sing(X')} \\
	{\Sing(L)} & {\Sing(Y)}
	\arrow["{\ell_*}", from=1-1, to=1-2]
	\arrow["{i''_*}"', hook, from=1-1, to=2-1]
	\arrow["{f''_*}", hook, from=1-2, to=2-2]
	\arrow["{v_*}"', from=2-1, to=2-2]
\end{tikzcd}
\end{equation}
where we've abbreviated $\Sing(g) = g_*$ for any morphism $g$. Note that $i''_*$ and $f''_*$ are monomorphisms as $\Sing$ is a right adjoint and hence preserves monomorphisms. We wish to show that this is a pushout square. Now pushouts are computed degreewise in $\ncat{sSet}$, so we need only prove that $\Sing(Y)_n \cong \Sing(L)_n +_{\Sing(L')_n} \Sing(X')_n$ for all $n \geq 0$. Now suppose that $\sigma : \bDelta^n \to Y$ is an element of $\Sing(Y)_n$. 

If the image of $\sigma$ does not intersect $X \subseteq Y$, then $\sigma$ must factor entirely through $L$. So suppose that $\sigma$ does intersect $X$. Then $\sigma \subseteq n_Y(X)$, and as shown above $n_Y(X) \subseteq X'$, so $\sigma \in X'$. Therefore the map $v_* + f''_* : \Sing(L)_n + \Sing(X')_n \to \Sing(Y)_n$ is surjective.

Now suppose that $\sigma: \bDelta^n \to L$ and $\tau : \bDelta^n \to X'$ are simplices such that $v\sigma = v_*(\sigma) = f''_*(\tau) = f'' \tau$. 

Since $v \sigma = f'' \tau$, this implies that the image of $ v\sigma$ is contained in $X'$ and $L$. Since $Y = L +_{L'} X'$, this implies that there exists a $\theta : \bDelta^n \to L'$ such that $i'' \theta = v \sigma$. But then
\begin{equation*}
    f'' \ell \theta = v i'' \theta = v \sigma = f'' \tau.
\end{equation*}
Since $f'' : X' \to Y$ is a monomorphism, this implies that $\ell \theta = \tau$. Hence $v_* + f''_*$ restricts to a bijection
\begin{equation*}
    \Sing(L)_n +_{\Sing(L')_n} \Sing(X')_n \xrightarrow{v_* + f''_*} \Sing(Y)_n.
\end{equation*}
Therefore (\ref{eq strict sing pushout}) is a pushout in $\ncat{sSet}$.

We obtain a commutative diagram
\begin{equation} \label{eq sing hoequiv of cospans}
    \begin{tikzcd}
	{\Sing(L)} & {\Sing(K)} & {\Sing(X)} \\
	{\Sing(L)} & {\Sing(L')} & {\Sing(X')}
	\arrow[equals, from=1-1, to=2-1]
	\arrow["{i_*}"', from=1-2, to=1-1]
	\arrow["{u_*}", from=1-2, to=1-3]
	\arrow["{i'_*}", from=1-2, to=2-2]
	\arrow["{f'_*}", from=1-3, to=2-3]
	\arrow["{i_*''}", from=2-2, to=2-1]
	\arrow["{\ell_*}"', from=2-2, to=2-3]
\end{tikzcd}
\end{equation}
Note that $i'$ and $m$ are $\times$-homotopy equivalences since they are inclusions of $\times$-deformation retracts, and therefore $i'_*$ and $f'_*$ are weak equivalences of simplicial sets by Proposition \ref{prop Sing preserves x-homotopy equivalence}.

Therefore (\ref{eq sing hoequiv of cospans}) is a weak equivalence of cospans, so taking homotopy pushouts gives a weak equivalence $e$ in the following commutative diagram
\begin{equation*}
    \begin{tikzcd}
	{\Sing(L) +^\text{ho}_{\Sing(K)} \Sing(X)} & {\Sing(L) +_{\Sing(K)} \Sing(X)} \\
	{\Sing(L) +^\text{ho}_{\Sing(L')} \Sing(X')} & {\Sing(Y)}
	\arrow["\simeq", from=1-1, to=1-2]
	\arrow["e"', from=1-1, to=2-1]
	\arrow[from=1-2, to=2-2]
	\arrow["\simeq"', from=2-1, to=2-2]
\end{tikzcd}
\end{equation*}
The bottom horizontal arrow is a weak equivalence since $\Sing$ preserved the bottom pushout square in (\ref{eq pasting pushouts Sing proof}), which is a pushout along a monomorphism $i''_*$, hence $\Sing(Y)$ is also a homotopy pushout. Similarly the top horizontal map above is a weak equivalence because $i_*$ is a monomorphism of simplicial sets, and hence is a homotopy pushout. Hence by $2$-of-$3$, the right hand vertical map above is a weak equivalence. Hence the commutative square (\ref{eq sing of original pushout square Sing proof}) is a homotopy pushout of simplicial sets.
\end{proof}

\begin{Rem}
Note that Proposition \ref{prop sing sends NDR pairs to homotopy pushouts} is extremely similar to \cite[Proposition 4.3]{thomason1980cat}. In that context, Thomason defines a notion of Dwyer map $f : \cat{C} \to \cat{D}$ between categories which essentially says that $f$ factors as a deformation retract along with some other compatibility conditions. What Proposition \ref{prop sing sends NDR pairs to homotopy pushouts} tells us is that in the context of simplicial complexes, the correct analog of Dwyer maps are inclusions of $\times$-NDR pairs. Proving that $N : \ncat{Cat} \to \ncat{sSet}$ sends pushouts along Dwyer maps to homotopy pushouts is the key argument to Thomason's proof. In this context, with simplicial complexes instead of categories, the proof is simpler.
\end{Rem}

The above proof along with Lemma \ref{lem pushout of deformation retract} also proves the following result.

\begin{Lemma} \label{lem pushout of ndr pair}
Suppose that
\begin{equation}
    \begin{tikzcd}
	K & X \\
	L & Y
	\arrow["u", from=1-1, to=1-2]
	\arrow["i"', hook, from=1-1, to=2-1]
	\arrow["f", hook, from=1-2, to=2-2]
	\arrow["v"', from=2-1, to=2-2]
	\arrow["\lrcorner"{anchor=center, pos=0.125, rotate=180}, draw=none, from=2-2, to=1-1]
\end{tikzcd}
\end{equation}
is a pushout diagram in $\ncat{Cpx}$ where $i : K \to L$ is the inclusion of an $\times$-NDR pair of simplicial complexes. Then $f : X \hookrightarrow Y$ is the inclusion of an $\times$-NDR pair.
\end{Lemma}

\begin{Lemma} \label{lem unit maps are weak equivs}
The unit maps
\begin{equation} \label{eq unit maps}
\begin{aligned}
& \Delta^n \xrightarrow{\eta_{\Delta^n}} \tSing \, \tRe(\Delta^n) \\
& \Lambda^n_k \xrightarrow{\eta_{\Lambda^n_k}} \tSing \, \tRe(\Lambda^n_k)
\end{aligned}
\end{equation}
are weak equivalences of simplicial sets.
\end{Lemma}

\begin{proof}
We first note that $\u{\Lambda}^n_k$ strongly collapses to the vertex $k \in V(\u{\Lambda}^n_k)$. Hence by Lemma \ref{lem sd preserves strong collapse}, $\Sd^2 \, \u{\Lambda}^n_k$ strongly collapses to $\Sd^2 \bDelta^0 \cong \bDelta^0$. Therefore the inclusion map $\bDelta^0 \xrightarrow{k} \Sd^2 \, \u{\Lambda}^n_k$ is an $\times$-homotopy equivalence by Lemma \ref{lem strong collapse is x-def retract}, and so by Proposition \ref{prop Sing preserves x-homotopy equivalence}, $k : \Delta^0 \cong \Sing \, \bDelta^0 \to \Sing \, \Sd^2 \, \u{\Lambda}^n_k$ is a weak equivalence of simplicial sets.
We have the following commutative diagram
\begin{equation*}
\begin{tikzcd}
	{\Delta^0} && {\Lambda^n_k} \\
	{\text{Sing} \, \Sd^2 \, \u{\Lambda}^n_k} & {\text{Ex}^2 \, \text{Sing} \, \Sd^2 \, \u{\Lambda}^n_k} & {\text{Ex}^2 \, \text{Sing} \, \text{Re}(\sd^2 \Lambda^n_k)}
	\arrow["\simeq", from=1-1, to=1-3]
	\arrow["\simeq"', from=1-1, to=2-1]
	\arrow["{\eta_{\Lambda^n_k}}", from=1-3, to=2-3]
	\arrow["\simeq", from=2-1, to=2-2]
	\arrow["\cong", from=2-2, to=2-3]
\end{tikzcd}
\end{equation*}
where the top horizontal map, the left vertical map and bottom left horizontal map are all weak equivalences of simplicial sets. The bottom right horizontal map is an isomorphism by Corollary \ref{cor realization of subdivision of horns, simplices, boundaries}. Hence by $2$-of-$3$, the unit map is a weak equivalence.

The proof of $\eta_{\Delta^n}$ being a weak equivalence follows similarly.
\end{proof}

Now we are ready to transfer the model structure.

\begin{Th} \label{th transferred model structure from sset}
The right-transferred model structure on $\ncat{Cpx}$ transferred along the adjunction (\ref{eq thomason adjunction cpx and sset}) from the Kan-Quillen model structure on $\ncat{sSet}$ exists.
\end{Th}

\begin{proof}
Let us begin verifying the conditions of Proposition \ref{prop right-transfer}.

(1) The category $\ncat{Cpx}$ is a quasitopos and hence locally presentable. Therefore by \cite[Remark 1.30, 2.15]{rosicky1994locally}, every object is small with respect to the whole category. In fact the domains of $\tRe(I)$ and $\tRe(J)$ are finite with respect to the whole category.

(2) The domains and codomains of the simplicial sets in $I$ and $J$ are finite simplicial sets, hence by Lemma \ref{lem transfinite composite of weak equivs}, transfinite composites of weak equivalences are weak equivalences in $\ncat{sSet}$.

(3) We note that $\text{Ex}$ preserves $\lambda$-sequences by \cite[Proof of Theorem 6.3 part (iii)]{fiore2010thomason}. The functor $\Sing$ preserves $\lambda$-sequences for the same reason
\begin{equation*}
\begin{aligned}
\Sing(\colim_{\beta \leq \lambda} K_\beta)_n & = \ncat{Cpx}(\bDelta^n, \colim_{\beta \leq \lambda} K_\beta) \\
& \cong \colim_{\beta \leq \lambda} \ncat{Cpx}(\bDelta^n, K_\beta) \\
& = \colim_{\beta \leq \lambda} \Sing(K)_n
\end{aligned}   
\end{equation*}
namely the $n$-simplex $\bDelta^n$ is a finite object in $\ncat{Cpx}$, and hence degreewise $\Sing$ preserves $\lambda$-sequences, but colimits are computed degreewise in $\ncat{sSet}$, hence the result follows.

Now let us prove condition (4). Let us write $\tRe = \Re \, \sd^2$ and $\tSing = \text{Ex}^2 \, \Sing$. Suppose we have a pushout diagram
\begin{equation} \label{eq pushout of realizations}
\begin{tikzcd}
	{\tRe(\Lambda^n_k)} & A \\
	{\tRe(\Delta^n)} & B
	\arrow[from=1-1, to=1-2]
	\arrow["{{\tRe(j)}}"', from=1-1, to=2-1]
	\arrow["{{j'}}", from=1-2, to=2-2]
	\arrow[from=2-1, to=2-2]
	\arrow["\lrcorner"{anchor=center, pos=0.125, rotate=180}, draw=none, from=2-2, to=1-1]
\end{tikzcd} 
\end{equation}
we want to show that $\tSing(j') : \tSing(A) \to \tSing(B)$ is a weak equivalence in $\ncat{sSet}$ (equivalently $\Sing(j')$ is a weak equivalence). 

By Lemma \ref{lem realization preserves subdivision}, $\tRe(j)$ is isomorphic to the inclusion map
\begin{equation*}
\Sd^2 \, \u{\Lambda}^n_k \hookrightarrow \Sd^2 \, \bDelta^n.
\end{equation*}
Hence by Proposition \ref{prop double subdivision is NDR pair}, $\tRe(j)$ is isomorphic to the inclusion of an NDR pair. Hence by Proposition \ref{prop sing sends NDR pairs to homotopy pushouts}, applying $\tSing$ to (\ref{eq pushout of realizations}) we obtain a homotopy pushout diagram in $\ncat{sSet}$
\begin{equation} \label{eq sing pushout square}
\begin{tikzcd}
	{\tSing \, \tRe(\Lambda^n_k)} & {\tSing(A)} \\
	{\tSing \, \tRe(\Delta^n)} & {\tSing(B)}
	\arrow[from=1-1, to=1-2]
	\arrow[hook, from=1-1, to=2-1]
	\arrow[hook, from=1-2, to=2-2]
	\arrow[from=2-1, to=2-2]
\end{tikzcd}
\end{equation}
Pasting this commutative square with the unit we obtain a commutative diagram
\begin{equation} \label{eq full pasting square}
\begin{tikzcd}
	{\Lambda^n_k} & {\tSing \, \tRe(\Lambda^n_k)} & {\tSing(A)} \\
	{\Delta^n} & {\tSing \, \tRe(\Delta^n)} & {\tSing(B)}
	\arrow["{{\eta_{\Lambda^n_k}}}", from=1-1, to=1-2]
	\arrow["\simeq"', hook, from=1-1, to=2-1]
	\arrow[from=1-2, to=1-3]
	\arrow[hook, from=1-2, to=2-2]
	\arrow["{\tSing(j')}", hook, from=1-3, to=2-3]
	\arrow["{{\eta_{\Delta^n}}}"', from=2-1, to=2-2]
	\arrow[from=2-2, to=2-3]
\end{tikzcd}
\end{equation}
By Lemma \ref{lem unit maps are weak equivs}, the unit maps $\eta_{\Lambda^n_k}$ and $\eta_{\Delta^n}$ are weak equivalences. Hence by $2$-of-$3$, each map in the left hand square of (\ref{eq full pasting square}) is a weak equivalence. But since the right hand square is a homotopy pushout, by Proposition 3.4.2.10 of \cite[\href{https://kerodon.net/tag/0112}{Tag 0112}]{kerodon}, this implies that $\tSing(j')$ is a weak equivalence.
\end{proof}

\section{Properties of the Thomason Model Structure} \label{section properties of thomason}

In the previous section, we proved that the right-transfer of the Kan-Quillen model structure to $\ncat{Cpx}$ exists. In this section, we look at this model structure in more detail. We call the model structure on $\ncat{Cpx}$ that exists by Theorem \ref{th transferred model structure from sset} the \textbf{Thomason model structure} on $\ncat{Cpx}$\footnote{It is unclear whether the model structures in this paper should really be credited to Thomason or Matsushita, hence we chose to credit both of them. We name the corresponding model structures on graphs in Section \ref{section graphs} after Matsushita.}. In summary, it is defined as the model structure where a map $f : K \to L$ of simplicial complexes is a
\begin{itemize}
    \item weak equivalence if and only if $\Sing(f) : \Sing(K) \to \Sing(L)$ is a weak equivalence of simplicial sets,
    \item fibration if and only if $ \tSing(f)$ is a Kan fibration,
    \item cofibration if it is a relative $\tRe(I)$-cell complex.
\end{itemize}
We refer to these as Thomason (co)fibrations and weak equivalences, and similarly define Thomason (co)fibrant simplicial complexes. We furthermore obtain the following basic properties of this model structure obtained by right-transfer.

\begin{Cor} \label{cor properties of thomason model structure on cpx}
The Thomason model structure on $\ncat{Cpx}$ is cofibrantly generated (by Theorem \ref{prop right-transfer}) and right proper (by \cite[Lemma 4.1]{nlab:transfer}). Furthermore the adjunction (\ref{eq thomason adjunction cpx and sset}) is a Quillen adjunction.
\end{Cor}

Corollary \ref{cor properties of thomason model structure on cpx} immediately supplies us with a wealth of cofibrant objects and cofibrations.

\begin{Cor} \label{cor cofibrant complexes}
Given a simplicial set $X$, the simplicial complex $\tRe X$ is Thomason cofibrant. In particular, if $X = \oSing(K_{\leq})$ is an
ordered simplicial complex, then
\begin{equation*}
   \tRe \, \oSing(K_\leq) = \Re \, \sd^2 \oSing(K_\leq) \cong \Sd^2 K 
\end{equation*}
is Thomason cofibrant. Furthermore, given any monomorphism $i : X \hookrightarrow Y$ of simplicial sets,
the map $\tRe(i)$ is a Thomason cofibration. In particular, if $i : K_\leq \hookrightarrow L_\leq$ is a monomorphism of ordered
simplicial complexes, then $\Sd^2(i) : \Sd^2 K \hookrightarrow \Sd^2 L$ is a Thomason cofibration.
\end{Cor}

\begin{Rem}
Note that Corollary \ref{cor properties of thomason model structure on cpx} gives many simple examples of cofibrant objects and cofibrations in the Thomason model structure on $\ncat{Cpx}$. Such examples are more complicated in the case of the Thomason model structure on $\ncat{Cat}$, as the functor $h : \ncat{sSet} \to \ncat{Cat}$ is more complicated than $\Re : \ncat{sSet} \to \ncat{Cpx}$. We will explore more consequences of this in Section \ref{section (co)fibrant objects}.
\end{Rem}

The following results will help us prove that (\ref{eq thomason adjunction cpx and sset}) is a Quillen equivalence. The following proofs are structurally similar to those of \cite[Section 3]{matsushita2017box}.

\begin{Lemma} \label{lem push out of weak equiv units is weak equiv unit}
Suppose that
\begin{equation} \label{eq pushout of simplicial sets}
\begin{tikzcd}
	X & A \\
	Y & B
	\arrow["u", from=1-1, to=1-2]
	\arrow["i"', hook, from=1-1, to=2-1]
	\arrow["f", hook, from=1-2, to=2-2]
	\arrow["v"', from=2-1, to=2-2]
	\arrow["\lrcorner"{anchor=center, pos=0.125, rotate=180}, draw=none, from=2-2, to=1-1]
\end{tikzcd}
\end{equation}
is a pushout of simplicial sets, where $i$ and $f$ are monomorphisms. If the unit maps
\begin{equation*}
X \xrightarrow{\eta_X} \tSing \tRe X, \qquad Y \xrightarrow{\eta_Y} \tSing \tRe Y, \qquad A \xrightarrow{\eta_A} \tSing \tRe A
\end{equation*}
are all weak equivalences of simplicial sets, then $\eta_B$ is a weak equivalence of simplicial sets.
\end{Lemma}

\begin{proof}
Since $i$ is a monomorphism, it is a cofibration in the Kan-Quillen model structure on $\ncat{sSet}$. Hence (\ref{eq pushout of simplicial sets}) is a homotopy pushout. Since the unit maps $\eta_X, \eta_Y, \eta_A$ are all weak equivalences, this implies that the map $\eta_Y +_{\eta_X} \eta_A$ in the commutative diagram
\begin{equation*}
\begin{tikzcd}
	X && A \\
	Y && {Y +_X A \cong B} \\
	&& {\tSing \tRe X} & {\tSing \tRe A} \\
	&& {\tSing \tRe Y} & {\tSing\tRe Y +_{\tSing \tRe X} \tSing \tRe A}
	\arrow[""{name=0, anchor=center, inner sep=0}, "u", from=1-1, to=1-3]
	\arrow["i"', hook, from=1-1, to=2-1]
	\arrow["{\eta_X}", from=1-1, to=3-3]
	\arrow["f", hook, from=1-3, to=2-3]
	\arrow["{\eta_A}", from=1-3, to=3-4]
	\arrow["v"', from=2-1, to=2-3]
	\arrow["{\eta_Y}"', from=2-1, to=4-3]
	\arrow["{\eta_{Y} +_{\eta_X} \eta_A}", from=2-3, to=4-4]
	\arrow["{u_*}"{pos=0.2}, from=3-3, to=3-4]
	\arrow["{i_*}"', hook, from=3-3, to=4-3]
	\arrow[hook, from=3-4, to=4-4]
	\arrow[from=4-3, to=4-4]
	\arrow["\lrcorner"{anchor=center, pos=0.125, rotate=180}, draw=none, from=2-3, to=0]
\end{tikzcd}
\end{equation*}
is a weak equivalence. Now the map
\begin{equation*}
\tRe K \cong \Sd^2 \Re K \xhookrightarrow{\Sd^2 \Re i} \Sd^2 \Re L \cong \tRe L
\end{equation*}
is an $\times$-NDR pair by Proposition \ref{prop double subdivision is NDR pair}. Now by Proposition \ref{prop sing sends NDR pairs to homotopy pushouts}, the top horizontal map in the commutative diagram
\begin{equation*}
    \begin{tikzcd}
	{\Sing\tRe Y +_{\Sing \tRe X} \Sing \tRe A} & {\Sing \tRe B} \\
	{\tSing\tRe Y +_{\tSing \tRe X} \tSing \tRe A} & {\tSing \tRe B}
	\arrow["\simeq", from=1-1, to=1-2]
	\arrow["{\rho^2_{\tRe Y} +_{\rho^2_{\tRe X}} \rho^2_{\tRe A}}"', from=1-1, to=2-1]
	\arrow["{\rho^2_{\tRe B}}", from=1-2, to=2-2]
	\arrow[from=2-1, to=2-2]
\end{tikzcd}
\end{equation*}
where $\rho_X : X \to \text{Ex} \, X$ is the adjunct of the last vertex map $\lambda_X : \sd X \to X$. The map $\rho_X$ is a trivial cofibration by \cite[Theorem 22]{moss2020another}, hence the left hand vertical map is a weak equivalence, and the right hand vertical map is a weak equivalence. Thus the bottom horizontal map is a weak equivalence. So by $2$-of-$3$, the unit map $\eta_B$ in the commutative diagram
\begin{equation*}
    \begin{tikzcd}
	B \\
	{\tSing\tRe Y +_{\tSing \tRe X} \tSing \tRe A} & {\tSing \tRe B}
	\arrow["{\eta_Y +_{\eta_X} \eta_A}"', from=1-1, to=2-1]
	\arrow["{\eta_B}", from=1-1, to=2-2]
	\arrow[from=2-1, to=2-2]
\end{tikzcd}
\end{equation*}
is a weak equivalence of simplicial sets.
\end{proof}

\begin{Th} \label{th cpx quillen equiv to sset}
The Quillen adjunction (\ref{eq thomason adjunction cpx and sset}) is a Quillen equivalence.
\end{Th}

\begin{proof}
Since the Thomason model structure is right-transferred along (\ref{eq thomason adjunction cpx and sset}) and hence $\tSing$ creates weak equivalences, by the dual of \cite[Lemma 3.3]{erdal2019} (\cite[Proposition 2.3]{nlab:quillen_equivalence}), to show that (\ref{eq thomason adjunction cpx and sset}) is a Quillen equivalence, it is enough to show that for every simplicial set $X$, the unit map
\begin{equation*}
    X \xrightarrow{\eta_X} \tSing \, \tRe(X)
\end{equation*}
is a weak equivalence of simplicial sets. But this follows by taking $\Gamma$ to be the trivial group in \cite[Lemma 3.4]{matsushita2017box}.
\end{proof}

Now let us investigate the Thomason cofibrations. Recall from Definition \ref{def cofibrantly generated model category} that the Thomason cofibrations are precisely the class $\text{cof}(\tRe(I))$, the retracts of relative $\tRe(I)$-cell complexes. We wish to show that these are all inclusions of $\times$-NDR pairs. We have already shown that inclusions of $\times$-NDR pairs are closed under pushouts in Lemma \ref{lem pushout of ndr pair}. 

\begin{Lemma}
Let $i : K \hookrightarrow L$, $j : L \hookrightarrow M$ be inclusions of $\times$-NDR pairs. Then $ji : K \hookrightarrow M$ is the inclusion of an $\times$-NDR pair.
\end{Lemma}

\begin{proof}
Since $i$ and $j$ are inclusions of $\times$-NDR pairs, there exist subcomplexes $A \subseteq L$, $B \subseteq M$ such that $n_L(K) \subseteq A$, $n_M(L) \subseteq B$ and $i_A : K \hookrightarrow A$, $j_B : L \hookrightarrow B$ are inclusions of $\times$-deformation retracts with retract maps $r_A : A \to K$ and $r_B : B \to L$ respectively.

Consider the diagram
\begin{equation*}
    \begin{tikzcd}
	K & {n_L(K)} & L & {n_{M}(L)} & M \\
	&& {n_M(K)}
	\arrow[hook, from=1-1, to=1-2]
	\arrow[hook, from=1-1, to=2-3]
	\arrow["{r|_{n_L(K)}}"', curve={height=18pt}, from=1-2, to=1-1]
	\arrow[hook, from=1-2, to=1-3]
	\arrow[hook, from=1-2, to=2-3]
	\arrow[hook, from=1-3, to=1-4]
	\arrow["{s|_{n_M(L)}}"', curve={height=18pt}, from=1-4, to=1-3]
	\arrow[hook, from=1-4, to=1-5]
	\arrow[hook, from=2-3, to=1-4]
	\arrow[hook, from=2-3, to=1-5]
\end{tikzcd}
\end{equation*}
it is easy to see that $n_L(K) \subseteq n_M(K)$ and $n_M(K) \subseteq n_M(L)$. Furthermore $s(n_M(K)) \subseteq n_L(K)$, because if $\sigma \in n_M(K)$, then there is some $v \in V(K)$ such that $\sigma \cup v \in n_M(L) \subseteq M$. Therefore 
\begin{equation*}
s(\sigma) \cup v = s(\sigma) \cup s(v) = s(\sigma \cup v) \in L.
\end{equation*}
Hence $s(\sigma) \in L$ is a simplex such that $s(\sigma) \cup v \in L$, so $s(\sigma) \in \st_L(v)$ and hence $s(\sigma) \in n_L(K)$. 

Let $H_A : A \times I_n \to A$ and $H_B : B \times I_m \to B$ denote the homotopies witnessing the $\times$-deformation retractions of $A$ onto $K$ and $B$ onto $L$ respectively. Now let us define a map $H : n_M(K) \times I_N \to n_M(K)$, with $N = n + m$ by
\begin{equation*}
    H(x,i) = \begin{cases}
        H_B(x, i), & 0 \leq i \leq n, \\
        H_A(s(x), i - n), & n \leq i \leq N.
    \end{cases}
\end{equation*}
Then we see that $H(x,0) = x$ for all $x \in n_M(K)$, $H(k,i) = k$ for all $0 \leq i \leq N$ and $H(x,N) = H_A(s(x),m) = rs(x)$. Hence $H$ is a $\times$-deformation retraction, and thus $ji : K \hookrightarrow M$ is the inclusion of a $\times$-NDR pair.
\end{proof}

We now obtain the following strengthening of Proposition \ref{prop sing sends NDR pairs to homotopy pushouts}.

\begin{Prop} \label{prop sing sends pushouts of cofibrations to homotopy pushouts}
Given a pushout square
\begin{equation*}
   \begin{tikzcd}
	K & A \\
	L & B
	\arrow["u", from=1-1, to=1-2]
	\arrow["i"', hook, from=1-1, to=2-1]
	\arrow["f", hook, from=1-2, to=2-2]
	\arrow["v"', from=2-1, to=2-2]
	\arrow["\lrcorner"{anchor=center, pos=0.125, rotate=180}, draw=none, from=2-2, to=1-1]
\end{tikzcd}    
\end{equation*}
in $\ncat{Cpx}$ where $i : K \hookrightarrow L$ is a Thomason cofibration, the commutative square
\begin{equation*}
    \begin{tikzcd}
	{\Sing(K)} & {\Sing(X)} \\
	{\Sing(L)} & {\Sing(Y)}
	\arrow["{u_*}", from=1-1, to=1-2]
	\arrow["{i_*}"', hook, from=1-1, to=2-1]
	\arrow["{f_*}", hook, from=1-2, to=2-2]
	\arrow["{v_*}"', from=2-1, to=2-2]
\end{tikzcd}   
\end{equation*}
is a homotopy pushout in the Kan-Quillen model structure on $\ncat{sSet}$.
\end{Prop}

\begin{proof}
We will follow the argument of Fjellbo \cite[Lemma 8.4]{fjellbo2020nonsingular}. Suppose that
\begin{equation*}
    \begin{tikzcd}
	K & A \\
	L & B
	\arrow["u", from=1-1, to=1-2]
	\arrow["i"', hook, from=1-1, to=2-1]
	\arrow["f", hook, from=1-2, to=2-2]
	\arrow["v"', from=2-1, to=2-2]
	\arrow["\lrcorner"{anchor=center, pos=0.125, rotate=180}, draw=none, from=2-2, to=1-1]
\end{tikzcd}
\end{equation*}
is a pushout square in $\ncat{Cpx}$, where $i$ is a cofibration. Then we know that $i$ is a retract of a $\tRe(I)$-cell complex. But here we can be more precise. Since the Thomason model structure on $\ncat{Cpx}$ is cofibrantly generated, we can factor $i = q j$ using the small object argument where $j : K \hookrightarrow E$ is a relative $\tRe(I)$-complex and $q : E \to L$ is a trivial fibration. Hence there is a solution to the following lifting problem
\begin{equation*}
    \begin{tikzcd}
	K & E \\
	L & L
	\arrow["j", hook, from=1-1, to=1-2]
	\arrow["i"', hook, from=1-1, to=2-1]
	\arrow["q", two heads, from=1-2, to=2-2]
	\arrow["s"{description}, dashed, from=2-1, to=1-2]
	\arrow[equals, from=2-1, to=2-2]
\end{tikzcd}
\end{equation*}
So $s$ exhibits $i$ as a retract of the relative $\tRe(I)$-complex $j : K \hookrightarrow E$. Hence we obtain the following commutative diagram
\begin{equation*}
\begin{tikzcd}
	{\Sing(K)} & {\Sing(A)} \\
	{\Sing(L)} & {\Sing(L)+_{\Sing(K)} \Sing(A)} & {\Sing(B)} \\
	{\Sing(E)} & {\Sing(E)+_{\Sing(K)} \Sing(A)} & {\Sing(E+_K A)} \\
	{\Sing(L)} & {\Sing(L)+_{\Sing(K)} \Sing(A)} & {\Sing(B)}
	\arrow["{u_*}", from=1-1, to=1-2]
	\arrow["{i_*}"', from=1-1, to=2-1]
	\arrow["{j_*}"', curve={height=24pt}, from=1-1, to=3-1]
	\arrow[from=1-2, to=2-2]
	\arrow[from=2-1, to=2-2]
	\arrow["{s_*}"', from=2-1, to=3-1]
	\arrow["{1_{\Sing(L)}}"', curve={height=24pt}, from=2-1, to=4-1]
	\arrow["\phi", from=2-2, to=2-3]
	\arrow[from=2-2, to=3-2]
	\arrow[from=2-3, to=3-3]
	\arrow["{1_{\Sing(B)}}", curve={height=-50pt}, from=2-3, to=4-3]
	\arrow[from=3-1, to=3-2]
	\arrow["{q_*}"', from=3-1, to=4-1]
	\arrow["\psi", from=3-2, to=3-3]
	\arrow[from=3-2, to=4-2]
	\arrow[from=3-3, to=4-3]
	\arrow[from=4-1, to=4-2]
	\arrow["\phi", from=4-2, to=4-3]
\end{tikzcd}
\end{equation*}
in which each of the left hand squares is a pushout. Now the map $j : K \hookrightarrow E$ is a relative $\tRe(I)$-cell complex, and hence is a transfinite composition of $\times$-NDR pairs. Hence applying $\Sing$, we see that every step of the transfinite composite is a homotopy pushout by Proposition \ref{prop sing sends NDR pairs to homotopy pushouts}. Since transfinite composites of weak equivalences of simplicial sets are weak equivalences by Lemma \ref{lem transfinite composite of weak equivs}, the map $\psi$ is a weak equivalence of simplicial sets. Hence $\phi$ is a retract of $\psi$, and therefore is a weak equivalence as well. Hence $\Sing$ sends the pushout to a homotopy pushout.
\end{proof}

\begin{Prop} \label{prop thomason model structure is proper}
The Thomason model structure on $\ncat{Cpx}$ is proper.
\end{Prop}

\begin{proof}
We have already shown that the Thomason model structure on  $\ncat{Cpx}$ is right proper by Corollary \ref{cor properties of thomason model structure on cpx}. Hence we just need to show that $\ncat{Cpx}$ is left proper. So suppose we have a pushout diagram
\begin{equation*}
 \begin{tikzcd}
	K & A \\
	L & B
	\arrow["u", from=1-1, to=1-2]
	\arrow["i"', hook, from=1-1, to=2-1]
	\arrow["f", hook, from=1-2, to=2-2]
	\arrow["v"', from=2-1, to=2-2]
	\arrow["\lrcorner"{anchor=center, pos=0.125, rotate=180}, draw=none, from=2-2, to=1-1]
\end{tikzcd} 
\end{equation*}
in $\ncat{Cpx}$ where $i$ is a Thomason cofibration and $u$ is a Thomason weak equivalence. Then by Proposition \ref{prop sing sends pushouts of cofibrations to homotopy pushouts} its image under $\Sing$
\begin{equation*}
    \begin{tikzcd}
	{\Sing(K)} & {\Sing(X)} \\
	{\Sing(L)} & {\Sing(Y)}
	\arrow["{u_*}", from=1-1, to=1-2]
	\arrow["{i_*}"', hook, from=1-1, to=2-1]
	\arrow["{f_*}", hook, from=1-2, to=2-2]
	\arrow["{v_*}"', from=2-1, to=2-2]
\end{tikzcd}    
\end{equation*}
is a homotopy pushout of simplicial sets, and $u_*$ is a weak equivalence of simplicial sets. Hence by Proposition 3.4.2.10 of \cite[\href{https://kerodon.net/tag/0112}{Tag 0112}]{kerodon}, $v_*$ is a weak equivalence of simplicial sets. Therefore $v : L \to Y$ is a Thomason weak equivalence.
\end{proof}

\begin{Def}
We say that a simplicial complex $K$ is a \textbf{flag complex} if it is in the essential image of the clique complex functor $\Cl : \ncat{Gr} \to \ncat{Cpx}$ from (\ref{eq clique complex adjunction}). Equivalently, a simplicial complex $K$ is flag if the following condition holds: given a subset $S \subseteq V(K)$, $S \in K$ if and only if for every pair $x,y \in S$, $\{x,y\} \in K$.
\end{Def}

\begin{Lemma} \label{lem subdivision is flag}
If $K$ is a simplicial complex, then $\Sd \, K$ is a flag complex.
\end{Lemma}

\begin{proof}
Recall that $\Sd \, K = (\Ord \circ \Face)(K)$, see Definition \ref{def subdivision}. Hence $\Face(K)$ is the poset of simplices ordered by inclusion, and thus the simplices of $\Sd \, K$ are those sets of simplices of $K$ that form a chain by inclusion. But a chain in $\Face(K)$ is precisely a complete subgraph in the $1$-skeleton $(\Sd \, K)_{\leq 1}$. Hence the simplices of $\Sd \, K$ are determined by its underlying graph. 
\end{proof}

\begin{Rem}
In fact, every simplicial complex $K$ of the form $K = \Ord(P)$ for a poset $P$ is a flag complex. This is because $K \cong \Cl(\text{Comp}(P))$, where $\text{Comp}(P)$ is the \textbf{comparibility graph} of the poset $P$. It is the graph with vertices the elements of $P$ and where $xy \in E(\text{Comp}(P))$ if and only if $x \leq y$ or $y \leq x$. So clique complexes and flag complexes are the same class of simplicial complexes. Flags are another name for chains in posets, hence the name flag complex.
\end{Rem}

\begin{Lemma} \label{lem pushouts of flag complexes are flag}
Suppose that
\begin{equation*}
\begin{tikzcd}
	K & A \\
	L & B
	\arrow["u", from=1-1, to=1-2]
	\arrow["i"', hook, from=1-1, to=2-1]
	\arrow["f", hook, from=1-2, to=2-2]
	\arrow["v"', from=2-1, to=2-2]
	\arrow["\lrcorner"{anchor=center, pos=0.125, rotate=180}, draw=none, from=2-2, to=1-1]
\end{tikzcd} 
\end{equation*}
is a pushout in $\ncat{Cpx}$, where $L$ and $A$ are flag simplicial complexes and $i$ is a monomorphism. Then $B$ is a flag simplicial complex.
\end{Lemma}

\begin{proof}
Let $S \subseteq V(B)$ be a subset such that for all $x, y \in S$, $\{x,y \} \in B$. Then the $1$-skeleton $S_{\leq 1}$ is a complete graph inside of $B$. We want to show that $S$ is a simplex in $B$. 

Let us identify $A$ with $f(A) \subseteq B$. Suppose $S \subseteq V(A)$. Then since $A$ is a flag complex, $S \in A$, and hence $S \in B$. 

Now suppose that $S \nsubseteq V(A)$. If $S$ lies wholly in $v(L)$, then since $B$ is a pushout, there must exist a complete graph $S'$ in $L$ such that $v(S') = S$. But $L$ is a flag complex, so $S'$ is a simplex in $L$, and since $v$ is a map of simplicial complexes, this implies that $S$ is a simplex in $B$.

So now suppose that there exists $x, y \in V(S)$ such that $\{x,y\} \in B$ and $x = [a]$ for some $a \in A$ and $y = [\ell]$ for some $\ell \in L$. But by the set-theoretic description of the pushout $B$, $\{x,y\} \in B$ is only possible if either (1) there exists a $k \in K$ such that $a = u(k)$ and $\{\ell, k \} \in L$, or (2) $\ell = i(k')$ for some $k' \in K$, $a' = u(k')$ and $\{a, a' \} \in A$.

In the case of (1), if $a = u(k)$ and $\{\ell, k \} \in L$, then for every other vertex $s \in V(S)$, since $S$ is a complete graph $s$ must be adjacent to $\ell$, so it must be the case that $s = [\ell']$ for $\ell' \in L$ or $s = [a']$ for $a' = u(k')$ for some $k' \in K$, in which case $s = [k']$. In other words, it is not possible for $s = [a'']$ for $a'' \in (A \setminus u(K))$, because there are no edges in $B$ connecting vertices of $(A \setminus u(K))$ with $L$. So if (1) holds, then every edge in $S$ must come from $L$. But $L$ is a flag complex, so $S$ is a simplex, and hence a simplex in $B$.

In the case of (2), if $\ell = i(k')$, $a' = u(k')$ and $\{a, a' \} \in A$, then for every other $s \in V(S)$, since $s$ must be connected with $x = [a]$, it must be the case that $s = [a'']$ for some $a'' \in A$ or $s = [k]$ for $k \in K$, in which case $s = [u(k)]$, because there are no edges in $B$ connecting vertices of $A$ with $(L \setminus K)$. So if (2) holds, then every edge in $S$ must come from $A$. But $A$ is a flag complex, so $S$ is a simplex, and hence a simplex in $B$.
\end{proof}

\begin{Lemma} \label{lem filtered colimit of flag complexes are flag}
Given an ordinal $\lambda$ and a $\lambda$-sequence $K : \lambda \to \ncat{Cpx}$
\begin{equation*}
    K_0 \hookrightarrow K_1 \hookrightarrow K_2 \hookrightarrow \dots \hookrightarrow K_\beta \hookrightarrow \dots
\end{equation*}
where all of the maps are monomorphisms and each $K_\beta$ is a flag complex, then the colimit
\begin{equation*}
    K = \ncolim{\beta < \lambda} K_\beta
\end{equation*}
is a flag complex.
\end{Lemma}

\begin{proof}
Suppose that $S \subseteq V(K)$ forms a complete graph in $K$. Then since $S$ is a finite simplicial complex, there exists some $\beta$ such that $S \subseteq K_\beta$. But $K_\beta$ is a flag complex, hence $S$ is a simplex in $K_\beta$, and therefore in $K$.
\end{proof}

\begin{Prop} \label{prop cofibrant complexes are flag}
Every Thomason cofibrant simplicial complex is a flag complex.
\end{Prop}

\begin{proof}
If $K$ is a Thomason cofibrant simplicial complex, then it is a retract of a transfinite composition of pushouts\footnote{Normally one says transfinite composition of coproducts of pushouts of maps in $\tRe(I)$. But this is redundant by \cite[Proposition 10.2.14]{hirschhorn2003model}.} of maps in $\tRe(I)$ from (\ref{eq gen cofibs for cpx}).

So let $K$ be an $\tRe(I)$-cell complex. In other words, let $K$ be constructed by transfinite compositions of pushouts of maps in $\tRe(I)$. By Lemma \ref{lem subdivision is flag}, every simplicial complex in $\tRe(I)$ is a flag complex, and $\varnothing$ is vacuously a flag complex, so by Lemma \ref{lem pushouts of flag complexes are flag} and Lemma \ref{lem filtered colimit of flag complexes are flag}, $K$ is a flag complex.

Now if $K' \xhookrightarrow{i} K$ is a retract of $K$, i.e. there exists a map $r : K \to K'$ such that $ri = 1_{K'}$, then $i$ is a split monomorphism, and hence is a regular monomorphism. But the regular monomorphisms in $\ncat{Cpx}$ are precisely the induced subcomplexes. But clearly any induced subcomplex of a flag complex is a flag complex. Hence $K'$ is a flag complex.
\end{proof}

\begin{Rem}
Note the similarity of Proposition \ref{prop cofibrant complexes are flag} with \cite[Proposition 5.7]{thomason1980cat} which says that every cofibrant category $\cat{C}$ in the Thomason model structure on $\ncat{Cat}$ is a poset.
\end{Rem}

\begin{Rem}
In \cite{thomason1980cat} it is shown that every cofibrant object in the Thomason model structure on $\ncat{Cat}$ is a poset. This begs the question, do there exists posets that are not cofibrant? In our context, in lieu of Proposition \ref{prop cofibrant complexes are flag}, the question is: do there exist flag complexes that are not cofibrant? In \cite[Proposition 6.2]{may2017poset} it is shown that there exists a poset $A$ that is not cofibrant in the Thomason model structure on posets. This poset $A$ has the property that $O = \text{Ord}(A)$ is the octohedron as a simplicial complex. However, because there does not exist a simplicial complex $K$ such that $A = \text{Face}(K)$, the methods used to prove this do not carry over straightforwardly to construct a non-cofibrant flag simplicial complex. It is therefore an open problem to determine if there exist non-cofibrant flag simplicial complexes in the Thomason model structure on $\ncat{Cpx}$.
\end{Rem}

\begin{Rem} \label{rem cpx not simplicial}
One can show, using the same example and technique as in \cite[Page 221]{raptis2010homotopy} that the Thomason model structure on  $\ncat{Cpx}$ is neither monoidal or simplicial. The same argument shows that the other model structures we will consider in the next section on $\ncat{Gr}$ and $\ncat{Gr}_\ell$ are also not monoidal or simplicial.
\end{Rem}

\begin{Lemma} \label{lem tsing is left quillen}
The right Quillen functor $\tSing$ is also left Quillen.
\end{Lemma}

\begin{proof}
Suppose that $i : K \to L$ is a Thomason cofibration of simplicial complexes. Then by definition it is a retract of a relative $\tRe(I)$-complex. Since $\tRe(I)$ consists of monomorphisms, and monomorphisms in $\ncat{Cpx}$ are closed under pushouts, transfinite composites and retracts, this implies that $i$ is a monomorphism. But $\tSing$ is a right adjoint, and hence preserves monomorphisms. Thus if $i$ is a Thomason cofibration, then $\tSing(i)$ is a monomorphism in $\ncat{sSet}$, equivalently a cofibration in the Kan-Quillen model structure. Furthermore, if $i$ is a weak equivalence, then by definition $\tSing(i)$ is a weak equivalence of simplicial sets.
\end{proof}

\section{Graphs} \label{section graphs}

Having right-transferred the Quillen model structure on $\ncat{sSet}$ to $\ncat{Cpx}$ to obtain the Thomason model structure, we will now work to further transfer this model structure to categories of graphs.

\begin{Def} \label{def loop graph}
A \textbf{loop graph} $G$ consists of a set $V(G)$ along with a symmetric relation $E(G) \subseteq V(G) \times V(G)$. We call the elements of $V(G)$ \textbf{vertices} and the elements of $E(G)$ \textbf{edges}. We write $xy \in E(G)$ to mean that $(x,y) \in E(G)$, and we say that $x$ and $y$ are \textbf{adjacent}. A graph homomorphism or graph map $f : G \to H$ of loop graphs consists of a function $V(f) : V(G) \to V(H)$ such that if $xy \in E(G)$, then $V(f)(x) V(f)(y) \in E(H)$. Let $\ncat{Gr}_\ell$ denote the category of loop graphs and graph homomorphisms.
\end{Def}

We can think of a loop graph as a collection of vertices with at most one edge between any pair of vertices. In particular this means that a loop graph is allowed to have at most one loop on every vertex. A morphism of loop graphs cannot crush an edge down to a vertex unless that vertex has a loop on it.

\begin{Def} \label{def reflexive graph}
We say a loop graph $G$ is \textbf{reflexive}, if the underlying edge relation $E(G)$ is reflexive. In other words, every vertex has a loop. We let $\ncat{Gr}$ denote the full subcategory of $\ncat{Gr}_\ell$ on the reflexive graphs\footnote{In previous work \cite{bumpus2025structured}, the notation $\ncat{Gr}$ meant the full subcategory of $\ncat{Gr}_\ell$ on the unlooped graphs. In this paper, since we only consider reflexive graphs and loop graphs, we feel justified in this abuse of notation.}.
\end{Def}

Hence a graph map $f : G \to H$ between reflexive graphs is allowed to crush an edge down to any vertex.

\subsection{The Matsushita Model Structure on Reflexive Graphs}

Now we transfer the Thomason model structure from $\ncat{Cpx}$ to obtain the Matsushita model structure on $\ncat{Gr}$.

We have an adjoint triple
\begin{equation} \label{eq clique complex triple adjunction}
\begin{tikzcd}
	{\ncat{Gr}} && {\ncat{Cpx}}
	\arrow[""{name=0, anchor=center, inner sep=0}, "\Cl"', curve={height=18pt}, hook, from=1-1, to=1-3]
	\arrow[""{name=1, anchor=center, inner sep=0}, "{\iota}", curve={height=-18pt}, hook', from=1-1, to=1-3]
	\arrow[""{name=2, anchor=center, inner sep=0}, "{(-)_{\leq 1}}"{description}, from=1-3, to=1-1]
	\arrow["\dashv"{anchor=center, rotate=-91}, draw=none, from=1, to=2]
	\arrow["\dashv"{anchor=center, rotate=-89}, draw=none, from=2, to=0]
\end{tikzcd}
\end{equation}
where $\Cl$ denotes the clique complex functor, which takes a graph $G$ to the simplicial complex $\Cl(G)$ where $V(\Cl(G)) = V(G)$ and $X = \{x_1, \dots, x_n \}$ is a simplex of $\Cl(G)$ if the induced subgraph on $X$ in $G$ is a clique/complete subgraph. Given a graph homomorphism $f : G \to H$, $\Cl(f) : \Cl(G) \to \Cl(H)$ is given by precisely the same map $V(f)$ on vertices. The functor $\iota$ is the fully faithful inclusion of reflexive graphs into simplicial complexes, and $(-)_{\leq 1}$ is the functor that takes the underlying graph of a simplicial complex by forgetting about its $n$-simplices for $n > 1$. It is easiest to see that $\Cl$ is fully faithful by noticing that the counit $\varepsilon_G : \Cl(G)_{\leq 1} \to G$ is an isomorphism.

Given reflexive graphs $G$ and $H$, let $H^G$ denote their internal hom in $\ncat{Gr}$. So $H^G$ is the graph with vertices given by functions $f : V(G) \to V(H)$ and two functions $f$, $f'$ are adjacent if for all pairs $xy \in E(G)$, $f(x)f'(y) \in E(H)$. Similarly for simplicial complexes $K$ and $L$, let $L^K$ denote the internal hom in $\ncat{Cpx}$. So $V(L^K)$ is the set of functions $f : V(K) \to V(L)$, and $\{f_0, \dots, f_n \} \in L^K$ if for every $\{x_0, \dots, x_n \} \in K$, $\{f_0(x_0), f_1(x_1), \dots, f_n(x_n) \} \in L$. Let us note the following convenient fact.

\begin{Lemma} \label{lem cl preserves internal hom}
Given reflexive graphs $G$ and $H$, there is an isomorphism
\begin{equation*}
    \Cl(H)^{\Cl(G)} \xrightarrow{\cong} \Cl(H^G).
\end{equation*}
\end{Lemma}

\begin{proof}
This follows from \cite[Corollary A.1.5.9]{Johnstone2002}, since $\Cl$ is fully faithful and its left adjoint $(-)_{\leq 1}$ preserves limits and hence products.
\end{proof}

\begin{Lemma}
Suppose that $f, g: G \to H$ are graph homomorphisms such that $f \simeq_\times g$ (in the sense of simplicial complexes), then $\Cl(f) \simeq_\times \Cl(g)$.
\end{Lemma}

\begin{proof}
If $h : G \times I_n \to H$ is an $\times$-homotopy, then since $\Cl$ is a right adjoint and $\Cl(I_n) \cong I_n$, we obtain a $\times$-homotopy
\begin{equation*}
    \Cl(h) : \Cl(G) \times I_n \to \Cl(H).
\end{equation*}
\end{proof}

\begin{Cor}
If $f : G \to H$ is an $\times$-homotopy equivalence between graphs, then so is $\Cl(f) : \Cl(G) \to \Cl(H)$.
\end{Cor}

To right-transfer the Thomason model structure on $\ncat{Cpx}$ to $\Gr$, it will be helpful to consider the composite adjunction
\begin{equation}
\begin{tikzcd}
	\Gr && {\ncat{sSet}}
	\arrow[""{name=0, anchor=center, inner sep=0}, "{{{\tSing \, \circ \, \Cl}}}"', shift right=3, from=1-1, to=1-3]
	\arrow[""{name=1, anchor=center, inner sep=0}, "{{{(-)_{\leq 1} \circ \, \tRe}}}"', shift right=3, from=1-3, to=1-1]
	\arrow["\dashv"{anchor=center, rotate=-90}, draw=none, from=1, to=0]
\end{tikzcd}
\end{equation}
Let $\Re_\Gr = (-)_{\leq 1} \circ \tRe$ and $\Sing_\Gr = \tSing \circ \Cl$. It is easy to see that if $G$ is a graph, then $\Sing_\Gr(G)_n \cong \Gr(K^{n+1}, G)$, where $K^n$ is the complete graph on $n$ vertices.

If $K$ is a simplicial complex, then $K_{\leq 1}$ is a graph, but we will sometimes abuse notation and allow $K_{\leq 1}$ to also denote the simplicial complex $\iota(K_{\leq 1})$.

\begin{Lemma} \label{lem 1-truncation preserves def retract}
Given a pair $(L, K)$ of simplicial complexes, if $K$ is an $\times$-deformation retract of $L$, then $K_{\leq 1}$ is an $\times$-deformation retract of $L_{\leq 1}$.
\end{Lemma}

\begin{proof}
Given a homotopy $H : K \times I_n \to K$ between $1_K$ and $ir$ for some retraction map $r : L \to K$, since $(-)_{\leq 1}$ is a right adjoint, we have a map $H_{\leq 1} : K_{\leq 1} \times I_n \to K_{\leq 1}$ such that $H_{\leq 1}(-,0) = 1_{K_{\leq 1}}$ and $H_{\leq 1}(-, n) = ir$.
\end{proof}

\begin{Lemma} \label{lem 1-truncation preserves x-NDR pairs}
If $(L, K)$ is an $\times$-NDR pair of simplicial complexes, then so is $(K_{\leq 1}, L_{\leq 1})$.
\end{Lemma}

\begin{proof}
First we note that $n_{L_{\leq 1}}(K_{\leq 1}) = (n_L(K))_{\leq 1}$. Since $(L, K)$ is an $\times$-NDR pair, there exists a subcomplex $L' \subseteq L$ such that $n_L(K) \subseteq L'$ and $L'$ $\times$-deformation retracts to $K$. Hence $n_{L_{\leq 1}}(K_{\leq 1}) = (n_L(K))_{\leq 1} \subseteq L'_{\leq 1}$ and $L'_{\leq 1} \subseteq L_{\leq 1}$. Since $K$ is an $\times$-deformation retract of $L'$, $K_{\leq 1}$ is a $\times$-deformation retract of $L$ by Lemma \ref{lem 1-truncation preserves def retract}.
\end{proof}

The proof of the following result is nearly the same as that for Proposition \ref{prop sing sends NDR pairs to homotopy pushouts}, and hence we leave it to the reader.

\begin{Lemma} \label{lem clique complex sends NDR pushouts to homotopy pushouts}
The clique complex functor sends pushouts along $\times$-NDR pairs to homotopy pushouts in the Thomason model structure on $\ncat{Cpx}$.
\end{Lemma}

\begin{Th} \label{th transfer from cpx to gr}
The Thomason model structure on $\ncat{Cpx}$ right-transferred along the bottom adjunction of (\ref{eq clique complex triple adjunction}) to $\ncat{Gr}$ exists.
\end{Th}

\begin{proof}
It is easy to see that the first three conditions of Proposition \ref{prop right-transfer} hold in this case using similar reasoning as in Theorem \ref{th transferred model structure from sset}. Now consider the classes of maps
\begin{equation}
\begin{aligned}
\tRe_\Gr(I) = \tRe(I)_{\leq 1} & = \{ \tRe(\partial \Delta^n)_{\leq 1} \to \tRe(\Delta^n)_{\leq 1} \, | \, n \geq 0 \} \cong \{ (\Sd^2 \partial \u{\Delta}^n)_{\leq 1} \hookrightarrow (\Sd^2 \bDelta^n)_{\leq 1} \, | \, n \geq 0 \} \\
\tRe_\Gr(J) = \tRe(J)_{\leq 1} & = \{ \tRe(\Lambda^n_0)_{\leq 1} \to \tRe(\Delta^n)_{\leq 1} \, | \, n \geq 0 \} \cong \{ (\Sd^2 \u{\Lambda}^n_0)_{\leq 1} \hookrightarrow (\Sd^2 \bDelta^n)_{\leq 1} \, | \, n \geq 0 \}
\end{aligned}
\end{equation}
We want to verify condition (4) of Proposition \ref{prop right-transfer}. Suppose that the following diagram is a pushout
\begin{equation} \label{eq pushout square for graphs}
    \begin{tikzcd}
	{(\Sd^2 \, \u{\Lambda}^n_0)_{\leq 1}} & G \\
	{(\Sd^2 \, \bDelta^n)_{\leq 1}} & H
	\arrow["u", from=1-1, to=1-2]
	\arrow["j"', hook', from=1-1, to=2-1]
	\arrow["f", from=1-2, to=2-2]
	\arrow["v"', from=2-1, to=2-2]
	\arrow["\lrcorner"{anchor=center, pos=0.125, rotate=180}, draw=none, from=2-2, to=1-1]
\end{tikzcd}
\end{equation}
in $\Gr$. But $\iota : \ncat{Gr} \to \ncat{Cpx}$ is a left adjoint, hence this is also a pushout in $\ncat{Cpx}$. By Lemma \ref{lem 1-truncation preserves x-NDR pairs}, we know that $((\Sd^2 \, \u{\Lambda}^n_0)_{\leq 1}, (\Sd^2 \, \bDelta^n)_{\leq 1})$ is an $\times$-NDR pair. Hence by Lemma \ref{lem clique complex sends NDR pushouts to homotopy pushouts}, the right-hand square in the following diagram
\begin{equation*}
\begin{tikzcd}
	{\Sd^2 \, \u{\Lambda}^n_0} & {\Cl((\Sd^2 \, \u{\Lambda}^n_0)_{\leq1})} & {\Cl(G)} \\
	{\Sd^2 \, \bDelta^n} & {\Cl((\Sd^2 \, \bDelta^n)_{\leq1})} & {\Cl(H)}
	\arrow["\cong", from=1-1, to=1-2]
	\arrow["j"', hook, from=1-1, to=2-1]
	\arrow[from=1-2, to=1-3]
	\arrow[hook, from=1-2, to=2-2]
	\arrow["\Cl(f)", from=1-3, to=2-3]
	\arrow["\cong"', from=2-1, to=2-2]
	\arrow[from=2-2, to=2-3]
\end{tikzcd}
\end{equation*}
is a homotopy pushout in the Thomason model structure on $\ncat{Cpx}$. Since $j$ is a Thomason weak equivalence, so is $\Cl(f)$, by Proposition 3.4.2.10 of \cite[\href{https://kerodon.net/tag/0112}{Tag 0112}]{kerodon}.
\end{proof}

We call the model structure on $\ncat{Gr}$ that exists by Theorem \ref{th transfer from cpx to gr} the \textbf{Matsushita model structure} on $\ncat{Gr}$. In summary it is defined as the model structure where a map $f : G \to H$ of graphs is a
\begin{itemize}
    \item weak equivalence if and only if $\Sing_\Gr(f)$ is a weak equivalence of simplicial sets,
    \item fibration if and only if $\tSing_\Gr(f)$ is a Kan fibration,
    \item cofibration if it is a relative $\tRe_\Gr(I)$-cell complex.
\end{itemize}

Furthermore the reflective adjunction
\begin{equation} \label{eq clique complex adjunction}
\begin{tikzcd}
	{\ncat{Gr}} && {\ncat{Cpx}}
	\arrow[""{name=0, anchor=center, inner sep=0}, "\Cl"', shift right=3, hook, from=1-1, to=1-3]
	\arrow[""{name=1, anchor=center, inner sep=0}, "{(-)_{\leq1}}"', shift right=3, from=1-3, to=1-1]
	\arrow["\dashv"{anchor=center, rotate=-90}, draw=none, from=1, to=0]
\end{tikzcd}    
\end{equation}
is a Quillen adjunction. 

\begin{Th} \label{th quillen equiv between cpx and gr}
The Quillen adjunction (\ref{eq clique complex adjunction}) is a Quillen equivalence.
\end{Th}

\begin{proof}
As in the proof of Theorem \ref{th cpx quillen equiv to sset}, since $\Cl$ creates weak equivalences we only need to show that for every cofibrant simplicial complex $K$, the unit
\begin{equation*}
    K \xrightarrow{\eta_K} \Cl(K_{\leq 1})
\end{equation*}
is a Thomason weak equivalence. But by Proposition \ref{prop cofibrant complexes are flag}, since $K$ is Thomason cofibrant, it is a flag complex. Therefore the unit is an isomorphism, and the result follows.
\end{proof}

The Matsushita model structure on $\ncat{Gr}$ inherits the properties of the Thomason model structure on $\ncat{Cpx}$.

\begin{Prop} \label{prop model structure on gr is cofib gen and proper}
The Matushita model structure on $\ncat{Gr}$ is cofibrantly generated and proper.
\end{Prop}

\begin{proof}
Cofibrant generation follows from Proposition \ref{prop right-transfer}, and properness follows using nearly the same argument as in Proposition \ref{prop sing sends pushouts of cofibrations to homotopy pushouts}, but with $\Sing$ replaced with $\Sing_\Gr$, and Proposition \ref{prop thomason model structure is proper}.
\end{proof}

\subsection{The Matsushita Model Structure on Loop Graphs}

\begin{Def}
By a \textbf{loop graph} $G$ we mean a set $V(G)$ and a symmetric relation $E(G) \subseteq V(G) \times V(G)$. We often write $xy \in E(G)$ to mean that $(x,y)$ equivalently $(y,x) \in E(G)$. A graph homomorphism $f : G \to H$ between loop graphs consists of a function $V(f) : V(G) \to V(H)$ such that if $xy \in E(G)$, then $f(x)f(y) \in E(H)$. Let $\ncat{Gr}_\ell$ denote the category of loop graphs and graph homomorphisms. In other words, this is the category whose objects are simple undirected graphs, but where now a single loop is allowed. Morphisms in this category cannot collapse edges to vertices unless that vertex is looped.
\end{Def}

\begin{Rem}
Like the categories $\ncat{Cpx}$ and $\ncat{Gr}$, the category $\ncat{Gr}_\ell$ is a quasitopos. See \cite[Appendix A]{bumpus2025structured} for more details.
\end{Rem}

We have an adjoint triple
\begin{equation} \label{eq adjunction between reflexive and loop graphs}
\begin{tikzcd}
	{\ncat{Gr}_\ell} && {\ncat{Gr}}
	\arrow[""{name=0, anchor=center, inner sep=0}, "{(-)^\circ}"', curve={height=18pt}, from=1-1, to=1-3]
	\arrow[""{name=1, anchor=center, inner sep=0}, "{(-)_\ell}", curve={height=-18pt}, from=1-1, to=1-3]
	\arrow[""{name=2, anchor=center, inner sep=0}, "{{i_\ell}}"{description}, hook', from=1-3, to=1-1]
	\arrow["\dashv"{anchor=center, rotate=-88}, draw=none, from=1, to=2]
	\arrow["\dashv"{anchor=center, rotate=-92}, draw=none, from=2, to=0]
\end{tikzcd}
\end{equation}
where $i_\ell$ denotes the inclusion of reflexive graphs into loop graphs by thinking of reflexive graphs as loop graphs where every vertex is looped, $(-)_\ell$ is the functor that makes a loop graph reflexive by adding all loops to vertices that are unlooped, and $(-)^\circ$ takes the maximal reflexive subgraph of a loop graph. Note that $\Gr_\ell$ is cartesian closed, with internal hom $H^G$ defined in precisely the same way as in $\ncat{Gr}$.

\begin{Rem} \label{rem max reflective subgraph functor not cartesian closed}
Given loop graphs $G$ and $H$, then in general
\begin{equation*}
    (H^G)^\circ \ncong (H^\circ)^{G^\circ}.
\end{equation*}
For example, let $G$ be the loop graph $K^2_u$, the unlooped complete graph on two vertices, and let $H$ be the reflexive discrete graph on two vertices. Then it is easy to see that $H^G \cong H$, and hence $(H^G)^\circ \cong H^G \cong H$, while $G^\circ = \varnothing$ and so $(H^\circ)^{G^\circ} \cong K^1$ is the terminal graph.
\end{Rem}

Let
\begin{equation*}
    \tRe_{\Gr_\ell} = i_\ell \circ \tRe_{\Gr} = i_\ell \circ (-)_{\leq 1} \circ \tRe, \qquad \tSing_{\Gr_\ell} = \tSing_{\Gr} \circ (-)^\circ = \tSing \circ \Cl \circ (-)^\circ.
\end{equation*}

\begin{Th} \label{th transfer from gr to loop gr}
The right-transferred model structure along the bottom adjunction in (\ref{eq adjunction between reflexive and loop graphs}) from the Matsushita model structure on $\ncat{Gr}$ to $\ncat{Gr}_\ell$ exists.
\end{Th}

\begin{proof}
As in the proof of Theorem \ref{th transfer from cpx to gr}, it is easy to check the first three conditions of Proposition \ref{prop right-transfer} hold. 

We see that
\begin{equation}
\tRe_{\Gr_\ell}(I) = \tRe_\Gr(I), \qquad 
    \tRe_{\Gr_\ell}(J) = \tRe_\Gr(J).
\end{equation}
We want to verify condition (4) of Proposition \ref{prop right-transfer}. Suppose that the following diagram is a pushout
\begin{equation}
    \begin{tikzcd}
	{(\Sd^2 \, \u{\Lambda}^n_0)_{\leq 1}} & G \\
	{(\Sd^2 \, \bDelta^n)_{\leq 1}} & H
	\arrow["u", from=1-1, to=1-2]
	\arrow["j"', hook', from=1-1, to=2-1]
	\arrow["f", from=1-2, to=2-2]
	\arrow["v"', from=2-1, to=2-2]
	\arrow["\lrcorner"{anchor=center, pos=0.125, rotate=180}, draw=none, from=2-2, to=1-1]
\end{tikzcd}
\end{equation}
in $\Gr_\ell$. Note that $u$ and $v$ must factor through $G^\circ$ and $H^\circ$ respectively, since $(\Sd^2 \u{\Lambda}^n_0)_{\leq 1}$ and $(\Sd^2 \bDelta^n)_{\leq 1}$ are reflexive graphs and looped vertices must map to looped vertices in $\ncat{Gr}_\ell$. Hence by the pushout pasting Lemma \ref{lem pushout pasting lemma}, the above pushout square factors as
\begin{equation*}
\begin{tikzcd}
	{(\Sd^2 \, \u{\Lambda}^n_0)_{\leq 1}} & {G^\circ} & G \\
	{(\Sd^2 \, \bDelta^n)_{\leq 1}} & {H^\circ} & H
	\arrow["u", from=1-1, to=1-2]
	\arrow["j"', hook', from=1-1, to=2-1]
	\arrow[hook, from=1-2, to=1-3]
	\arrow["{f^\circ}", from=1-2, to=2-2]
	\arrow["f", from=1-3, to=2-3]
	\arrow["v"', from=2-1, to=2-2]
	\arrow["\lrcorner"{anchor=center, pos=0.125, rotate=180}, draw=none, from=2-2, to=1-1]
	\arrow[hook, from=2-2, to=2-3]
	\arrow["\lrcorner"{anchor=center, pos=0.125, rotate=180}, draw=none, from=2-3, to=1-2]
\end{tikzcd}
\end{equation*}
But then $f^\circ$ is a map in $\Gr$, and thus by the proof of Theorem \ref{th transfer from cpx to gr} is a Matsushita weak equivalence.
\end{proof}

As before, we call the model strucutre on $\ncat{Gr}_\ell$ that exists by Theorem \ref{th transfer from gr to loop gr} the \textbf{Matsushita model structure} on $\ncat{Gr}_\ell$. In summary it is defined as the model structure where a map $f : G \to H$ of loop graphs is a
\begin{itemize}
    \item weak equivalence if and only if $\Sing_{\Gr_\ell}(f)$ is a weak equivalence of simplicial sets,
    \item fibration if and only if $ \tSing_{\Gr_\ell}(f)$ is a Kan fibration,
    \item cofibration if it is a relative $\tRe_{\Gr_\ell}(I)$-cell complex.
\end{itemize}

Furthermore the coreflective adjunction
\begin{equation} \label{eq loop graph adjunction}
\begin{tikzcd}
	{\ncat{Gr}_\ell} && {\ncat{Gr}}
	\arrow[""{name=0, anchor=center, inner sep=0}, "{(-)^\circ}"', shift right=3, from=1-1, to=1-3]
	\arrow[""{name=1, anchor=center, inner sep=0}, "{i_\ell}"', shift right=3, hook', from=1-3, to=1-1]
	\arrow["\dashv"{anchor=center, rotate=-90}, draw=none, from=1, to=0]
\end{tikzcd} 
\end{equation}
is a Quillen adjunction.

\begin{Th} \label{th quillen equiv between gr and loop gr}
The Quillen adjunction (\ref{eq loop graph adjunction}) is a Quillen equivalence.
\end{Th}

\begin{proof}
As in the proof of Theorem \ref{th cpx quillen equiv to sset}, since $(-)^\circ$ creates weak equivalences we only need to show that for every cofibrant reflexive graph $G$, the unit
\begin{equation*}
    G \xrightarrow{\eta_G} (i_\ell(G))^\circ
\end{equation*}
is a Matsushita weak equivalence. But $\eta_G$ is an isomorphism, so the result follows.
\end{proof}

\begin{Prop} \label{prop model structure loop graphs is proper}
The Matsushita model structure on $\ncat{Gr}_\ell$ is cofibrantly generated and proper.
\end{Prop}

\begin{proof}
Cofibrant generation follows from Proposition \ref{prop right-transfer} and properness follows using nearly the same argument as in Proposition \ref{prop sing sends pushouts of cofibrations to homotopy pushouts}, but with $\Sing$ replaced with $\Sing_{\Gr_\ell}$, and Proposition \ref{prop thomason model structure is proper}.
\end{proof}

\begin{Rem}
The model structure on $\ncat{Gr}_\ell$ proved to exist by Theorem \ref{th transfer from gr to loop gr} is precisely the model structure that Matsushita constructs in \cite{matsushita2017box} by right-transferring along the composite adjunction
\begin{equation*}
    \begin{tikzcd}
	{\ncat{Gr}_\ell} && {\ncat{sSet}}
	\arrow[""{name=0, anchor=center, inner sep=0}, "{\tSing_{\Gr_\ell}}"', shift right=3, from=1-1, to=1-3]
	\arrow[""{name=1, anchor=center, inner sep=0}, "{\tRe_{\Gr_\ell}}"', shift right=3, from=1-3, to=1-1]
	\arrow["\dashv"{anchor=center, rotate=-90}, draw=none, from=1, to=0]
\end{tikzcd}
\end{equation*}
\end{Rem}

\begin{Lemma} \label{lem adjunction between loop graphs is biQuillen}
The left Quillen functor $i_\ell$ is also right Quillen.
\end{Lemma}

\begin{proof}
If $f : G \to H$ is a (trivial) fibration in $\ncat{Gr}$, then $i_\ell(f)$ is a (trivial) fibration in $\ncat{Gr}_\ell$ if and only if $(i_\ell(f))^\circ$ is a (trivial) fibration in $\ncat{Gr}$, but $(i_\ell(f))^\circ = f$.
\end{proof}

\begin{Rem} \label{rem unlooped sing}
Note that $\Sing_{\Gr_\ell} : \Gr_\ell \to \ncat{sSet}$ does not ``see'' unlooped graphs. This is because we are probing loop graphs with reflexive complete graphs. We could instead probe our loop graphs with complete unlooped graphs $K^n_u$. However, $K^\bullet_u$ is not a simplicial object in $\ncat{Gr}_\ell$, because codegeneracy maps no longer exist, as they require squashing edges to vertices. Instead $K^\bullet_u$ forms a co-semisimplicial object in $\Gr_\ell$. Hence we obtain an alternate, unlooped $\Sing^u_{\Gr_\ell} : \ncat{Gr}_\ell \to \ncat{ssSet}$, where $\ncat{ssSet}$ is the category of semisimplicial sets, i.e. the presheaf category on $\mathsf{\Delta}^\text{inj}$, the subcategory of $\mathsf{\Delta}$ with only injective maps. There are some subtleties regarding abstract homotopical structures on $\ncat{ssSet}$, see \cite{nlab:model_structure_on_semi-simplicial_sets}, but transferring such homotopical structure along this functor to $\Gr_\ell$ deserves further investigation.
\end{Rem}

\section{(Co)Fibrant Objects and (Co)Fibrations} \label{section (co)fibrant objects}
In this section we explore the question of (co)fibrancy of several classes of objects in  the Matsushita model structures on $\ncat{Gr}_\ell$ and $\ncat{Gr}$ and the Thomason model structure on $\ncat{Cpx}$.

\begin{Lemma} \label{lem simplices are cofibrant}
Every simplex $\bDelta^n$ is Thomason cofibrant.
\end{Lemma}

\begin{proof}
By Corollary \ref{cor cofibrant complexes}, $\Sd^2 \bDelta^n$ is cofibrant. Let us consider $\bDelta^n$ as an ordered simplicial complex with its canonical total ordering. There is an inclusion $i : \bDelta^n \to \Sd^2 \bDelta^n$ given by 
\begin{equation*}
0 \mapsto \{0\}, \quad 1 \mapsto \{0,01 \}, \quad \, \dots, \, \quad  n \mapsto \{0, 1, \, \dots, \, 01\dots n \}  
\end{equation*}
and the iterated last vertex map $\lambda^2 = \lambda \circ \Sd(\lambda) : \Sd^2 \bDelta^n \to \bDelta^n$ acts a retraction $\lambda^2 i = 1_{\Delta^n}$. Hence $\bDelta^n$ is a retract of a cofibrant object and is therefore cofibrant.
\end{proof}

\begin{Lemma} \label{lem simplices are thomason fibrant}
Every simplex $\bDelta^n$ is Thomason fibrant.
\end{Lemma}

\begin{proof}
To show that $\bDelta^n$ is fibrant, it is enough to show that every map $\Sd^2 \u{\Lambda}^n_0 \to \bDelta^n$ extends to a map $\Sd^2 \bDelta^n$. But this is equivalent to extending a function $V(\Sd^2 \u{\Lambda}^n_k) \to \{0, 1, \dots, n \}$ to a function $V(\Sd^2 \bDelta^n) \to \{0, 1, \dots, n\}$, which obviously can always be done.
\end{proof}

Now let us note that if a simplicial complex $K$ is Thomason cofibrant, then the graph $K_{\leq 1}$ is Matsushita cofibrant in $\ncat{Gr}$ since $(-)_{\leq 1}$ is left Quillen. Furthermore $K_{\leq 1}$ is (co)fibrant in $\ncat{Gr}_\ell$ since $i_\ell$ is left and right Quillen by Lemma \ref{lem adjunction between loop graphs is biQuillen}. Hence we say that a reflexive graph $G$ is Matsushita (co)fibrant to mean it is (co)fibrant in $\Gr$ and hence in $\Gr_{\ell}$. We say that a reflexive graph $G$ is Thomason cofibrant when $G$ is Thomason cofibrant considered as a simplicial complex. In this case, $G$ is then Matsushita cofibrant as well. Of course, we obtain a large class of Matsushita cofibrant graphs immediately.

\begin{Lemma} \label{lem double subdivisions of graphs are cofibrant}
Given any simplicial set $X$, the graph $\tRe_{\Gr}(X)$ is Matsushita cofibrant.
In particular, given a graph $G$, the graph $\Sd^2 G$ is Matsushita cofibrant. Furthermore given any monomorphism $i : G \hookrightarrow H$ of graphs, the map $\Sd^2(i) : \Sd^2 G \hookrightarrow \Sd^2 H$ is a Matsushita cofibration.
\end{Lemma}

\begin{Lemma}
If $G$ is an unlooped graph, then it is Matushita fibrant in $\ncat{Gr}_\ell$.
\end{Lemma}

\begin{proof}
If $G$ is an unlooped graph, then $\Sing_{\Gr}(G) = \varnothing$, which is trivially fibrant as a simplicial set.
\end{proof}

\begin{Lemma}
If $G$ is an unlooped graph, then it is not Matsushita cofibrant in $\ncat{Gr}_\ell$.
\end{Lemma}

\begin{proof}
The cofibrant objects in $\ncat{Gr}_\ell$ are retracts of $\tRe_{\ncat{Gr}_\ell}(I)$-cell complexes. But $\tRe_{\ncat{Gr}_\ell}(I)$-cell complexes must be reflexive graphs, and a reflexive graph cannot map to an unlooped graph, hence $G$ cannot be a retract of a $\tRe_{\ncat{Gr}_\ell}(I)$-cell complex.
\end{proof}

\begin{Lemma} \label{lem complete graphs are fibrant}
Every reflexive complete graph $K^n$ is Matsushita fibrant.
\end{Lemma}

\begin{proof}
This follows from practically the same argument as Lemma \ref{lem simplices are thomason fibrant}.
\end{proof}

Recall the $n$-path reflexive graph $I_n$ from Definition \ref{def homotopy of simplicial complexes}.

\begin{Lemma} \label{lem paths are cofibrant}
For every $n \geq 1$, the $n$-path reflexive graph $I_n$ and the terminal graph $K^1$ are Thomason cofibrant, and the inclusion\footnote{If we label $I_n$ with vertices $0, \dots, n$, then the inclusions $0 : K^1 \hookrightarrow I_n$ and $n : K^1 \hookrightarrow I_n$ are isomorphic} $i: K^1 \hookrightarrow I_n$ of the terminal graph into an endpoint of $I_n$ is a cofibration.
\end{Lemma}

\begin{proof}
It is easy to see that $\Sd \, I_n \cong I_{2n}$. Hence $\Sd^2(I_n) \cong I_{4n}$ is cofibrant by Lemma \ref{lem double subdivisions of graphs are cofibrant}. Now without loss of generality suppose that $i : K^1 \to I_n$ maps the single vertex to $0$. We have the retraction
\begin{equation*}
\begin{tikzcd}
	K^1 & K^1 & K^1 \\
	{I_n} & {I_{4n}} & {I_n}
	\arrow[equals, from=1-1, to=1-2]
	\arrow["i"', hook, from=1-1, to=2-1]
	\arrow[equals, from=1-2, to=1-3]
	\arrow["{\Sd^2(i)}"', hook, from=1-2, to=2-2]
	\arrow["i"', hook, from=1-3, to=2-3]
	\arrow["j"', hook, from=2-1, to=2-2]
	\arrow["r"', from=2-2, to=2-3]
\end{tikzcd}
\end{equation*}
where $j$ sends $0, \dots, n$ to themselves in $I_{4n}$ and $r$ is the map that sends $0, \dots, n$ to themselves and $n+1, \dots, 4n$ to $n$. Hence $i$ is a cofibration and since the retraction of maps also gives a retraction of objects, $I_n$ is cofibrant.
\end{proof}

\begin{Cor}
All reflexive discrete graphs are Matsushita cofibrant and fibrant.
\end{Cor}

\begin{proof}
Coproducts of cofibrant objects are cofibrant, so discrete graphs are cofibrant by Lemma \ref{lem paths are cofibrant}, and they are fibrant since for any discrete graph $G$, $\Sing_{\Gr}(G)$ is a Kan complex.
\end{proof}

\begin{Cor} \label{cor trees are cofibrant}
Every reflexive forest\footnote{A forest is a coproduct of trees, i.e. a reflexive graph without any cycles.} $T$ is Thomason cofibrant. Furthermore, any inclusion of a tree into another tree is a Thomason cofibration.
\end{Cor}

\begin{proof}
Every tree $T$ can be written as a transfinite composition of pushouts of the form
\begin{equation*}
\begin{tikzcd}
	K^1 & T_0 \\
	{I_1} & {T_0 + e}
	\arrow["v", from=1-1, to=1-2]
	\arrow[from=1-1, to=2-1]
	\arrow[from=1-2, to=2-2]
	\arrow[from=2-1, to=2-2]
	\arrow["\lrcorner"{anchor=center, pos=0.125, rotate=180}, draw=none, from=2-2, to=1-1]
\end{tikzcd}
\end{equation*}
where $T_0 + e$ denotes $T_0$ with an extra edge glued onto a leaf node of $T_0$. Hence $T$ is cofibrant, and furthermore every tree inclusion $i : T \hookrightarrow T'$ can be obtained by pushing out along $K^1 \hookrightarrow I_1$, hence $i$ is a cofibration. Since coproducts of cofibrant objects are cofibrant, this implies that all forests are cofibrant.
\end{proof}

\begin{Lemma}
Every map of reflexive graphs of the form $I_n \to K^1$ is a Thomason trivial fibration.
\end{Lemma}

\begin{proof}
This can either be seen using the lifting characterization or using $\Sing$.
\end{proof}

\begin{Lemma}
Every surjective map $p : \bDelta^n \to \bDelta^m$ between simplices is a Thomason trivial fibration. 
\end{Lemma}

\begin{proof}
We see that $\Sing(p)$ is isomorphic to $Np' : NE[n] \to NE[m]$ where $p'$ is the obvious induced functor $p' : E[n] \to E[m]$ between indiscrete categories. Then $Np'$ is a trivial Kan fibration if and only if $p'$ is surjective on objects and fully faithful, \cite[Exercise 25.14]{rezk2022introduction}. This is the case for every surjective map $p$, hence $\Sing(p)$ is a trivial Kan fibration.
\end{proof}

\begin{Cor}
Every surjective map $p : K^n \to K^m$ between complete reflexive graphs is a Matsushita trivial fibration.
\end{Cor}

\begin{Lemma} \label{lem 4n-cycles are cofibrant}
Every reflexive $n$-cycle $C_n$ for $n \geq 3$ is Matsushita cofibrant.
\end{Lemma}

\begin{proof}
First consider the inclusion $\Sd^2 \partial \u{\Delta}^1 \to \Sd^2 \u{\Delta}^1$, which is isomorphic to $K^1 + K^1 \xhookrightarrow{(0,4)} I_4$. Consider the pushout 
\begin{equation*}
    \begin{tikzcd}
	{K^1 + K^1} & {I_m} \\
	{I_4} & {C_{m+4}}
	\arrow["{(0,m)}", from=1-1, to=1-2]
	\arrow["{(0,4)}"', from=1-1, to=2-1]
	\arrow[from=1-2, to=2-2]
	\arrow[from=2-1, to=2-2]
	\arrow["\lrcorner"{anchor=center, pos=0.125, rotate=180}, draw=none, from=2-2, to=1-1]
\end{tikzcd}
\end{equation*}
Since $I_n$ is cofibrant by Lemma \ref{lem paths are cofibrant}, it follows that $C_{m+4}$ is cofibrant for all $m \geq 0$. The $3$-cycle $C_3 = K^3$ is Matsushita cofibrant since $K^3 = (\u{\Delta}^3)_{\leq 1}$ and $(-)_{\leq 1}$ is left Quillen and simplices are Thomason cofibrant by Lemma \ref{lem simplices are cofibrant}.
\end{proof}

\section{Derived Hom} \label{section derived hom}

In this short section, we discuss derived mapping spaces in the Matsushita model structure on $\ncat{Gr}_\ell$. To any model category $M$, there exists a $\ncat{sSet}$-enriched category $\H(M)$, called the \textbf{hammock localization} \cite{dwyer1980calculating}, which is a model for the underlying $\infty$-category of the model category $M$. A nice write up of these ideas is given in \cite[Appendix A]{mazel2016quillen} and \cite{riehl2020homotopical}.

So, to every pair of objects $X, Y \in M$, there is an associated simplicial set which we denote by $\H(M)(X,Y)$, called the \textbf{derived mapping space} between $X$ and $Y$, that has the property that if there are weak equivalences $f : X \to X'$ or $g : Y \to Y'$ in $M$, then there are associated weak equivalences
\begin{equation*}
    \H(M)(f,Y) : \H(M)(X',Y) \to \H(M)(X,Y), \qquad \H(X, g) : \H(X,Y) \to \H(X,Y').
\end{equation*}

However, it is a nontrivial task to compute these derived mapping spaces. If $M$ is a simplicial model category, then the task is made easier by using the simplicial enrichment. However, none of the model structures on $\ncat{Gr}_\ell$, $\ncat{Gr}$ or $\ncat{Cpx}$ considered in this paper are simplicial, see Remark \ref{rem cpx not simplicial}. The next best way to compute derived mapping spaces in model categories is to use (co)simplicial resolutions or framings \cite[Chapters 16 and 17]{hirschhorn2003model}. However the author has been unable to prove that the obvious cosimplicial resolution of a simplicial complex $K$, given by $K \times \Sd^2 \bDelta^\bullet$ is Reedy cofibrant. Instead, we will leverage the fact that the Thomason and Matsushita model structures are all Quillen equivalent to the Kan-Quillen model structure on $\ncat{sSet}$, where it is easier to compute derived mapping spaces.

Given simplicial sets $X$ and $Y$, let $\u{\ncat{sSet}}(X,Y)$ denote the internal hom in $\ncat{sSet}$. We also denote this with the usual exponential notation $Y^X$. It is the simplicial set defined degreewise by
\begin{equation*}
\u{\ncat{sSet}}(X,Y)_n = \ncat{sSet}(X \times \Delta^n, Y).
\end{equation*}
Since $\ncat{sSet}$ is a simplicial model category, if $X$ is a cofibrant simplicial set and $Y$ is fibrant, equivalently $X$ is any simplicial set and $Y$ is a Kan complex, then
\begin{equation*}
    \H(\ncat{sSet})(X,Y) \simeq \u{\ncat{sSet}}(X,Y).
\end{equation*}

By Section A.2 of \cite{mazel2016quillen}, for any simplicial set $X$ and graph $G$, we have
\begin{equation*}
    \H(\ncat{sSet})(X, \tSing_\Gr(G)) \simeq \H(\ncat{Gr})(\tRe_\Gr(X), G).
\end{equation*}
Moreover for any pair of simplicial sets we have
\begin{equation*}
    \H(\ncat{sSet})(X,Y) \simeq \H(\ncat{Gr})(\tRe(X), \tRe(Y)),
\end{equation*}
and for any pair of Matsushita fibrant loop graphs we have
\begin{equation*}
 \H(\ncat{Gr}_\ell)(G,H) \simeq \H(\ncat{sSet})(\tSing_{\Gr_\ell}(G), \tSing_{\Gr_\ell}(H)).
\end{equation*}
In fact, as long as $H$ is Matsushita fibrant, then
\begin{equation*}
\begin{aligned}
    \H(\ncat{sSet})(\tSing_{\Gr_\ell}(G), \tSing_{\Gr_\ell}(H)) &\cong \u{\ncat{sSet}}(\tSing_{\Gr_\ell}(G), \tSing_{\Gr_\ell}(H)) \\
    & \simeq \u{\ncat{sSet}}(\Sing_{\Gr_\ell}(G), \tSing_{\Gr_\ell}(H)).
\end{aligned}
\end{equation*}

As mentioned in the Introduction, the author's original motivation in this work was towards understanding the Hom-complex of Lov\'{a}sz. Recall that by \cite[Remark 3.6]{dochtermann2009hom} the homotopy type of the Hom-complex can be modelled by a simplicial complex
\begin{equation*}
    \Hom(G,H) \simeq \Cl((H^G)^\circ).
\end{equation*}
In the case of reflexive graphs, this simplifies by the following result.

\begin{Lemma}
If $G$ and $H$ are reflexive graphs, then
\begin{equation*}
    \Sing \Hom(G,H) \simeq \Sing(\Cl(H)^{\Cl(G)}),
\end{equation*}
where the right-hand side is the internal hom in $\ncat{Cpx}$.
\end{Lemma}

\begin{proof}
This follows from the fact that $H^G$ is reflexive and Lemma \ref{lem cl preserves internal hom}.
\end{proof}

The author originally naively hoped that there would exist a weak equivalence
\begin{equation*}
    \H(\ncat{Gr}_\ell)(G, H) \simeq \Sing \, \Hom(G,H),
\end{equation*}
in general, but the following counterexample shows this to be incorrect. Let $G = K^2_u$ and $H = K^1 + K^1$ be the loop graphs from Remark \ref{rem max reflective subgraph functor not cartesian closed}. Both of these graphs are Matsushita fibrant, so we have
\begin{equation*}
    \H(\Gr_\ell)(G,H) \simeq \u{\ncat{sSet}}(\Sing_{\Gr_\ell}(G), \tSing_{\Gr_\ell}(H)).
\end{equation*}
In fact, $\Sing_{\Gr_\ell}(H) \cong \Delta^0 + \Delta^0$ is already a Kan complex, so we have
\begin{equation*}
    \H(\ncat{Gr}_\ell)(G,H) \simeq \Sing_{\Gr_\ell}(H)^{\Sing_{\Gr_\ell}(G)} \cong (\Delta^0 + \Delta^0)^{\varnothing} \cong \Delta^0.
\end{equation*}

Whereas
\begin{equation*}
    \Sing \, \Hom(G,H) \simeq \Sing_{\Gr_\ell}(H^G) \cong \Delta^0 + \Delta^0.
\end{equation*}

That being said, we do have the following obvious relationship.

\begin{Lemma}
Given loop graphs $G$ and $H$ such that $H^G$ is Matsushita fibrant, then
\begin{equation*}
    \Sing \, \Hom(G,H) \simeq \H(\ncat{Gr}_\ell)(K^1, H^G).
\end{equation*}
\end{Lemma}

\begin{proof}
If $H^G$ is Matsushita fibrant, then since the terminal graph $K^1$ is also Matsushita fibrant, we have
\begin{equation*}
\begin{aligned}
    \H(\ncat{Gr}_\ell)(K^1, H^G) & \simeq \u{\ncat{sSet}}(\Delta^0, \tSing_{\Gr_\ell}(H^G)) \\
    & \cong \tSing_{\Gr_\ell}(H^G) \\
    & \simeq \Sing_{\Gr_\ell}(H^G).
\end{aligned}
\end{equation*}
\end{proof}

Hence with $G$ and $H$ as above, we see that the derived mapping space does not respect the cartesian closure of the underlying category $\ncat{Gr}_\ell$,
\begin{equation*}
    \H(\Gr_\ell)(G,H) \cong \H(\Gr_\ell)(K^1 \times G, H) \nsimeq \H(\ncat{Gr}_\ell)(K^1, H^G).
\end{equation*}
This is a consequence of $\ncat{Gr}_\ell$ (along with $\ncat{Gr}$ and $\ncat{Cpx}$) not being a monoidal model category.

\appendix

\section{The Categories of (Ordered) Simplicial Complexes} \label{section categories}

\subsection{Simplicial Complexes}

While the description of $\ncat{Cpx}$ from Definition \ref{def cpx} is adequate, it will be helpful to characterize $\ncat{Cpx}$ more abstractly. We introduce symmetric sets for this purporse.

\begin{Def}
Let $\ncat{FinSet}_\times$ denote the category of nonempty finite sets and all functions between them. We let $\u{n} = \{1, 2, \dots, n \}$. Let $\ncat{SymSet} = \ncat{Set}^{\ncat{FinSet}_\times^\op}$, we call the objects of this category \textbf{symmetric sets}\footnote{They are also referred to as symmetric simplicial sets in the literature.}. If $X$ is a symmetric set, we let $X_n = X(\u{n + 1})$ for $n \geq 0$\footnote{This convention is in accordance with the labeling for simplicial sets.}. 
\end{Def}

The category of symmetric sets is obviously very nice, being a presheaf category. For our purposes, we wish to consider a certain full subcategory of $\ncat{SymSet}$.

\begin{Def}
We say that a small category $\cat{C}$ is \textbf{terminally concrete} if it has a terminal object $*$, and such that the functor
\begin{equation*}
    \cat{C}(*, -) : \cat{C} \to \ncat{Set}
\end{equation*}
is faithful. If $\cat{C}$ is terminally concrete, then a presheaf $X : \cat{C}^\op \to \ncat{Set}$ is \textbf{concrete} if for every $U \in \cat{C}$, the canonical map
\begin{equation*}
    X(U) \to \ncat{Set}(\cat{C}(*,U), X(*))
\end{equation*}
is injective.
\end{Def}

Let us note that $\ncat{FinSet}_\times$ is terminally concrete, with terminal object $\u{1}$.

\begin{Prop}[{\cite[Lemma 47]{baez2011convenient}}] \label{prop concrete presheaves form reflective subcat}
Given a terminally concrete category $\cat{C}$, the category of concrete presheaves on $\cat{C}$ is a reflective subcategory of presheaves on $\cat{C}$
\begin{equation*}
\begin{tikzcd}
	{\ncat{ConPre}(\cat{C})} && {\ncat{Pre}(\cat{C})}
	\arrow[""{name=0, anchor=center, inner sep=0}, "i"', shift right=3, hook, from=1-1, to=1-3]
	\arrow[""{name=1, anchor=center, inner sep=0}, "{{\text{Con}}}"', shift right=3, from=1-3, to=1-1]
	\arrow["\dashv"{anchor=center, rotate=-90}, draw=none, from=1, to=0]
\end{tikzcd}
\end{equation*}
\end{Prop}

The previous result tells us much about the (co)limits in $\ncat{ConPre}(\cat{C})$.

\begin{Prop}[{\cite[Proposition 4.5.15]{riehl2017category}}] \label{prop (co)limits in reflective subcategories}
Given a reflective subcategory $i: \cat{D} \hookrightarrow \cat{C}$ with left adjoint $L$ then
\begin{itemize}
    \item The inclusion $i : \cat{D} \hookrightarrow \cat{C}$ creates all limits that $\cat{C}$ admits.
    \item $\cat{D}$ has all colimits that $\cat{C}$ admits, formed by applying $L$ to the colimit of the diagram included in $\cat{C}$
\end{itemize}
Therefore if $\cat{D} \hookrightarrow \cat{C}$ is a reflective subcategory, with $\cat{C}$ (co)complete, then so is $\cat{D}$. Furthermore limits in $\cat{D}$ agree with those in $\cat{C}$, while colimits in $\cat{D}$ are computed by applying $L$ to the colimit computed in $\cat{C}$.
\end{Prop}

It was noticed by Grandis in \cite[Section 1.3]{Grandis2001higher} that simplicial complexes embed fully faithfully as the concrete presheaves into $\ncat{SymSet}$\footnote{Grandis calls these simple presheaves but mentions Lawvere's characterization of concrete presheaves.}. This was later proved more completely by Baez-Hoffnung in \cite[Proposition 27]{baez2011convenient}. We sketch the proof of this statement as the next result.

\begin{Lemma} \label{lem cpx is concrete presheaves}
The category $\ncat{Cpx}$ of simplicial complexes is equivalent to the category of concrete presheaves on $\ncat{FinSet}_\times$.
\end{Lemma}

\begin{proof}[Proof Sketch]
Given a simplicial complex $K$, let $\widetilde{K} : \ncat{FinSet}_\times^\op \to \ncat{Set}$ be the functor where $\widetilde{K}(\u{1}) = V(K)$ and $\widetilde{K}(\u{n})$ is the set of functions $p : \u{n} \to V(K)$ such that $\text{im} \, p \in K$. If $f : \u{n} \to \u{m}$ is a function, and $p : \u{m} \to V(K)$ is a plot of $\widetilde{K}$, then $\widetilde{K}(f) : \widetilde{K}(m) \to \widetilde{K}(n)$ is defined by precomposition $\widetilde{K}(f)(p) = pf$. Since $K$ is a simplicial complex, its simplices are downward closed, so this is well-defined. 

Conversely, given a concrete presheaf $X$ on $\ncat{FinSet}_\times$, let $K_X$ be the simplicial complex with $V(K_X) = X(\u{1})$ and $K_X = \{ \text{im} \, p \, | \, p \in X(\u{n}), \, n \geq 1\}$ where we are thinking of $p \in X(\u{n})$ equivalently as a function $p : \u{n} \to X_0 = X(\u{1})$.
\end{proof}

The category of concrete presheaves on a terminally concrete category inherits many nice properties of a presheaf topos, but not all of them. Categories of concrete presheaves are in particular Grothendieck quasitoposes, which implies that they are locally presentable, cartesian closed and have strong subobject classifiers, see \cite{baez2011convenient, garner2012grothendieck}.

\begin{Cor} \label{cor cpx is quasitopos}
The category $\ncat{Cpx}$ is a Grothendieck quasitopos.
\end{Cor}

Let us denote the corresponding adjunction between symmetric sets and simplicial complexes by
\begin{equation} \label{eq concreteness adjunction}
\begin{tikzcd}
	{\ncat{Cpx}} && {\ncat{SymSet}}
	\arrow[""{name=0, anchor=center, inner sep=0}, "i"', shift right=3, hook, from=1-1, to=1-3]
	\arrow[""{name=1, anchor=center, inner sep=0}, "{{{\text{Con}}}}"', shift right=3, from=1-3, to=1-1]
	\arrow["\dashv"{anchor=center, rotate=-90}, draw=none, from=1, to=0]
\end{tikzcd}
\end{equation}
Thus by Proposition \ref{prop (co)limits in reflective subcategories}, we know that $\ncat{Cpx}$ has the same limits as $\ncat{SymSet}$, and colimits are obtained by applying $\text{Con}$ to the colimit in $\ncat{SymSet}$.

Let us note that there is a triple adjunction
\begin{equation}
\begin{tikzcd}
	{\ncat{Cpx}} && {\ncat{Set}}
	\arrow[""{name=0, anchor=center, inner sep=0}, "V"{description}, from=1-1, to=1-3]
	\arrow[""{name=1, anchor=center, inner sep=0}, "{\text{CoDisc}}", curve={height=-18pt}, from=1-3, to=1-1]
	\arrow[""{name=2, anchor=center, inner sep=0}, "{\text{Disc}}"', curve={height=18pt}, from=1-3, to=1-1]
	\arrow["\dashv"{anchor=center, rotate=-90}, draw=none, from=0, to=1]
	\arrow["\dashv"{anchor=center, rotate=-90}, draw=none, from=2, to=0]
\end{tikzcd}
\end{equation}
where given a simplicial complex $K$, $V$ is the faithful functor where $V(K) = K_0$ is the set of vertices. Given a set $S$, $\text{Disc}(S)$ is the discrete simplicial complex on $S$, with only constant plots $p : \u{n} \to \u{1} \to \text{Disc}(S)$, and $\text{CoDisc}(S)$ is the codiscrete simplicial complex $\text{CoDisc}(S) = \bDelta^S$ where every function $p : \u{n} \to V(K)$ is a plot. If $S = \u{n+1}$ for $n \geq 0$, then $\bDelta^n = \text{CoDisc}(\u{n+1})$. 

\begin{Cor} \label{cor monos and epis of simplicial complexes}
A morphism $f : K \to L$ of simplicial complexes is a monomorphism/epimorphism if and only if $V(f)$ is injective/surjective. Furthemore the regular monomorphisms $ f : K \to L$ of simplicial complexes are precisely the induced subcomplexes, i.e. $f$ is a monomorphism, and if $S \subseteq V(K)$ with $f(S) \in L$, then $S \in K$.
\end{Cor}

\begin{proof}
Since $V$ is a right and left adjoint, it preserves monomorphisms and epimorphisms. So $V$ sends mono/epimorphisms to inj/surjections. Furthermore if $f : K \to L$ is a map of simplicial complexes such that $V(f)$ is inj/surjective, then $f$ is a mono/epimorphism, since $V$ is faithful, and hence reflects mono/epimorphisms. 

Now given a monomorphism $f : K \to L$, consider the cokernel pair $L \rightrightarrows L +_K L$. By the description of limits below and pushouts in Example \ref{ex pushouts along mono of simplicial complexes}, the equalizer of the cokernel pair is the simplicial complex with vertices $V(K)$, but where $S \subseteq V(K)$ is a simplex if and only if it is a simplex in $L$. In other words, the equalizer is the induced subcomplex of $L$ on $V(K)$. Hence a monomorphism of simplicial complexes is regular if and only if it is the inclusion of an induced subcomplex.
\end{proof}

Hence a plot $p : \u{n+1} \to V(K)$ of a simplicial complex $K$ is equivalently a map $p : \bDelta^n \to K$ of simplicial complexes. Hence we can fully understand a simplicial complex by using its sets of $n$-plots for $n \geq 0$. As a consequence of Lemma \ref{lem cpx is concrete presheaves} we obtain the following important result, which is often called the \textbf{Density Lemma} or \textbf{CoYoneda Lemma}. Given a simplicial complex $K$, let $(\bDelta \downarrow K)$ denote the category whose objects are maps $p : \bDelta^n \to K$, and given $q : \bDelta^m \to K$, a morphism $f : p \to q$ in $(\bDelta \downarrow K)$ is a map $f : \bDelta^n \to \bDelta^m$ of simplicial complexes such that $qf = p$.

\begin{Lemma} \label{lem coyoneda for simplicial complexes}
Given a simplicial complex $K$, there is an isomorphism
\begin{equation}
K \cong \ncolim{\bDelta^n \to K} \bDelta^n \cong \colim \left( (\bDelta \downarrow K) \xrightarrow{U} \ncat{Cpx} \right)
\end{equation}
where $U : (\bDelta \downarrow K) \to \ncat{Cpx}$ is the forgetful functor $U(p : \bDelta^n \to K) = \bDelta^n$.
\end{Lemma}

\begin{proof}
Since $\ncat{SymSet}$ is a presheaf topos, if $X$ is a symmetric set, then by the CoYoneda lemma for presheaf toposes, 
\begin{equation*}
    X \cong \ncolim{\bDelta^n \to X} \bDelta^n
\end{equation*}
where the colimit here is taken in $\ncat{SymSet}$, see \cite[Example 1.4.6]{riehl2014categorical}, \cite[Lemma A.75]{minichiello2025coverages} for more on the CoYoneda lemma. Now the functor $\text{Con}: \ncat{SymSet} \to \ncat{Cpx}$ from Proposition \ref{prop concrete presheaves form reflective subcat} is a left adjoint, and $\text{Con}(\bDelta^n) \cong \bDelta^n$, hence 
\begin{equation*}
\begin{aligned}
    K & \cong \text{Con}\, i K \\
    & \cong \text{Con} \left( \ncolim{\bDelta^n \to iK} \bDelta^n \right) \\
    & \cong \ncolim{\bDelta^n \to i K} \text{Con}\, \bDelta^n \\
    & \cong \ncolim{\bDelta^n \to K} \bDelta^n.
\end{aligned}
\end{equation*}
\end{proof}

We can describe (co)limits in $\ncat{Cpx}$ more concretely. (Co)limits in $\ncat{Cpx}$ are computed by first taking the corresponding (co)limit of the vertices in $\ncat{Set}$, and then giving the appropriate structure on the simplices. In more detail, if $X : I \to \ncat{Cpx}$ is a diagram, then $\lim_{i \in I} X(i)$ is the simplicial complex with $V(\lim_{i \in I} X(i)) = \lim_{i \in I} V(X(i))$ and limit cone maps $p_i : \lim_{i \in I} V(X(i)) \to V(X(i))$. A subset $S \subseteq V(\lim_{i \in I} X(i))$ is a simplex if and only if for every $i \in I$, the subsets $p_i(S) \subseteq X(i)$ are simplices.

Similarly $\colim_{i \in I} X(i)$ is the simplicial complex with $V(\colim_{i \in I} X(i)) = \colim_{i \in I} V(X(i))$ and colimit cocone maps $c_i : V(X(i)) \to \colim_{i \in I} V(X(i))$. A subset $S \subseteq V(\colim_{i \in I} X(i))$ is a simplex if and only if for some $i \in I$, there is an $i \in I$ and a simplex $T \in X(i)$ such that $c_i(T) = S$.

\begin{Ex}
For example, a pullback 
\begin{equation*}
    \begin{tikzcd}
	K & A \\
	L & B
	\arrow["u", from=1-1, to=1-2]
	\arrow["g"', from=1-1, to=2-1]
	\arrow["\lrcorner"{anchor=center, pos=0.125}, draw=none, from=1-1, to=2-2]
	\arrow["f", from=1-2, to=2-2]
	\arrow["v"', from=2-1, to=2-2]
\end{tikzcd}
\end{equation*}
in $\ncat{Cpx}$ is computed by first taking $V(K) = V(L) \times_{V(B)} V(A)$, and declaring that $S \subseteq V(K)$ is a simplex if and only if $g(S)$ and $u(S)$ are simplices in $L$ and $A$.   
\end{Ex}

\begin{Ex} \label{ex pushouts along mono of simplicial complexes}
Pushouts of simplicial complexes are also important for us. We will mostly be concerned with the case when we are pushing out along a monomorphism
\begin{equation*}
\begin{tikzcd}
	K & A \\
	L & B
	\arrow["u", from=1-1, to=1-2]
	\arrow["g"', hook, from=1-1, to=2-1]
	\arrow["f", from=1-2, to=2-2]
	\arrow["v"', from=2-1, to=2-2]
	\arrow["\lrcorner"{anchor=center, pos=0.125, rotate=180}, draw=none, from=2-2, to=1-1]
\end{tikzcd}
\end{equation*}
in which case $V(B) = V(L) +_{V(K)} V(A)$ is the disjoint union $V(L) + V(A)$ quotiented by the relation $\sim$, where for $\ell \in V(L)$ and $a \in V(A)$, $\ell \sim a$ if and only if there exists a finite zig-zag of elements of the form
\begin{equation*}
\begin{tikzcd}
	& {k_0} && {k_1} && {k_2} && {k_n} \\
	\ell && {a_0} && {\ell_1} && \dots && a
	\arrow["g"', maps to, from=1-2, to=2-1]
	\arrow["u", maps to, from=1-2, to=2-3]
	\arrow["u"', maps to, from=1-4, to=2-3]
	\arrow["g", maps to, from=1-4, to=2-5]
	\arrow["g"', maps to, from=1-6, to=2-5]
	\arrow["u", maps to, from=1-6, to=2-7]
	\arrow["g"', maps to, from=1-8, to=2-7]
	\arrow["u", maps to, from=1-8, to=2-9]
\end{tikzcd}    
\end{equation*}
i.e. $g(k_0) = \ell$, $u(k_0) = a_0$, $u(k_1) = a_0$, $\dots$, $u(k_n) = a$, or of the form starting with $a$ on the left.

Now since $g$ is injective on vertices, this implies that $k_1 = k_2$, which implies that $a_0 = u(k_1) = u(k_2) = a_1$, etc. Hence $\ell \sim a$ if and only if there exists a $k_0 \in V(K)$ such that $\ell = g(k_0)$ and $u(k_0) = a$. Note however, that if both $\ell, a \in V(L)$ or $\ell, a \in V(A)$, then there can be arbitrary length zig-zags that identify them in $V(L +_K A)$. A subset $S \subseteq V(L +_K A)$ is a simplex in $L +_K A$ if there exists a simplex $T$ in $L$ or $A$ such that $g(T) = S$ or $u(T) = S$.
\end{Ex}

The following well-known result, called the \textbf{pushout pasting lemma} from category theory will also be useful.

\begin{Lemma}[{\cite[Lemma 1.1]{lack2005adhesive}}] \label{lem pushout pasting lemma}
Given a category $\cat{C}$ with pushouts and a commutative diagram
\begin{equation*}
    \begin{tikzcd}
	A & C & E \\
	B & D & F
	\arrow[from=1-1, to=1-2]
	\arrow[from=1-1, to=2-1]
	\arrow[from=1-2, to=1-3]
	\arrow[from=1-2, to=2-2]
	\arrow[from=1-3, to=2-3]
	\arrow[from=2-1, to=2-2]
	\arrow[from=2-2, to=2-3]
\end{tikzcd}
\end{equation*}
if the left square is a pushout, then the outer rectangle is a pushout if and only if the right square is a pushout.
\end{Lemma}

\subsection{Ordered Simplicial Complexes} \label{section ordered simplicial complexes}

Recall the definition of the category $\ncat{oCpx}$ of ordered simplicial complexes from Definition \ref{def ordered complexes}.

Given ordered simplicial complexes $K_{\leq}, L_{\leq}$, with $K = U(K_{\leq})$ and $L = U(L_\leq)$, their product $K_{\leq} \times L_{\leq}$ is the ordered simplicial complex with $V(K_{\leq} \times L_{\leq}) = V(K) \times V(L)$ and where $\{(k_1, \ell_1), \dots, (k_n, \ell_n) \}$ forms a simplex if and only if $\{k_1, \dots, k_n \}$ and $\{\ell_1, \dots, \ell_n\}$ are totally ordered simplices in $K$ and $L$ respectively \cite[Proposition 1.25]{ruschoff2017lecture}. 

\begin{Ex} \label{ex ordered simplicial complex products}
Consider the ordered simplicial complex $K_{\leq}$ with $V(K) = \{0 \leq 1 \}$ and $K_{\leq} = \{ \{0 \}, \{1 \}, \{0, 1 \} \}$, and let $K = U(K_{\leq})$. Then $K_{\leq} \times K_{\leq}$ is the simplicial complex drawn below
\begin{equation*}
    \begin{tikzcd}
	{(0,1)} & {(1,1)} \\
	{(0,0)} & {(1,0)}
	\arrow[from=1-1, to=1-2]
	\arrow[from=2-1, to=1-1]
	\arrow[from=2-1, to=1-2]
	\arrow[from=2-1, to=2-2]
	\arrow[from=2-2, to=1-2]
\end{tikzcd}
\end{equation*}
with two $2$-simplices $\{(0,0), (0,1), (1,1)\}$ and $\{(0,0), (1,0), (1,1) \}$, whereas $K \times L$ is the simplicial complex $\u{\Delta}^3$.
\end{Ex}

Given an ordered simplicial complex $K_{\leq}$, we can describe $\oSing(K_{\leq})$ (Definition \ref{def osing}), alternatively as follows. If $S = \{x_0, \dots, x_n \}$ is an $n$-simplex in $K_\leq$, then we let $(x_0, \dots, x_n)$ also denote $S$, but remembering the total order $x_0 \leq x_1 \leq \dots \leq x_n$.

We let 
\begin{equation*}
 \oSing(K_\leq)_n = \{(x_0, \dots, x_n) \, | \, x_i \in V(K_{\leq}), \, (x_0, \dots, x_n) \in K_\leq \}
\end{equation*}
with face maps
\begin{equation*}
    d_i : \oSing(K_\leq)_n \to \oSing(K_\leq)_{n-1}, \, 0 \leq i \leq n, \qquad d_i(x_0, \dots, x_n) = (x_0, \dots, x_{i-1}, x_{i+1}, \dots, x_n)
\end{equation*}
and degeneracy maps
\begin{equation*}
    s_i : \oSing(K_\leq)_n \to \oSing(K_\leq)_{n+1}, \, 0 \leq i \leq n, \qquad s_i(x_0, \dots, x_n) = (x_0, \dots, x_i, x_i, \dots, x_n).
\end{equation*}

\begin{Rem}
By a semisimplicial set we mean a functor $ X: \bDelta_{\text{inj}}^\op \to \ncat{Set}$, where $\bDelta_\text{inj}$ is the subcategory of $\bDelta$ with the same objects but only injective order preserving maps. These are also called $\Delta$-complexes in \cite{hatcher2002algebraic}. Given an ordered simplicial complex $K_\leq$, the simplicial set $\oSing(K_\leq)$ is the semisimplicial set with face maps defined as above, with freely generated degeneracies.
\end{Rem}

\begin{Rem}
Shortly after the preprint version of this paper was put on the arXiv, the preprint \cite{wei2025ordered} appeared, putting a model structure on the category of ordered simplicial complexes. In a previous version of this paper, we mentioned \cite[Proposition 10.2.2]{may2003finite} which says that $\oSing$ is fully faithful. However \cite[Remark 5.3, 5.4]{wei2025ordered} disproves this. We recommend the reader to look at \cite{wei2025ordered} for more information on the category $\ncat{oCpx}$ and its relationship to simplicial sets.
\end{Rem}

\subsection*{Statements and Declarations}

\subsubsection*{Author Contributions}
Not applicable.

\subsubsection*{Funding}

The author was supported by Professional Staff Congress (PSC) CUNY grant TRADB 56-13.

\subsubsection*{Availability of Data and Materials}

No data or materials were used.

\subsubsection*{Conflict of Interest}
The author declares no competing interests.

\subsubsection*{Ethical Approval}
Not applicable.

\printbibliography
\end{document}